\documentclass[12pt]{amsart}
\usepackage{amscd,amsfonts,amsthm,amssymb,latexsym,amsmath, amscd, amsmath,amsxtra}
\usepackage[all]{xy}
\usepackage{anysize}
\usepackage[sc]{mathpazo} 
\usepackage{soul,xcolor}
\usepackage[toc,page]{appendix}
\usepackage[francais,english]{babel}
\usepackage[utf8]{inputenc}   
\usepackage{hyperref}
\usepackage{stmaryrd}
\usepackage{fancyhdr}
\usepackage{url}
\usepackage{dsfont}

\marginsize{2.5cm}{2.5cm}{2.5cm}{2.5cm}
\newtheorem{theorem}{Theorem}[section] 
\newtheorem{lemma}[theorem]{Lemma}
\newtheorem{corollary}[theorem]{Corollary}
\newtheorem{proposition}[theorem]{Proposition}
\theoremstyle{definition}

\newtheorem{example}{Example}
\newtheorem{definition}{Definition}

\newtheorem{question}{Question}
\theoremstyle{remark}
\newtheorem{remark}{Remark}

\renewcommand{\d}{\delta}
\newcommand{\R}{\mathbb{R}}
\newcommand{\GL}{\mathrm{GL}}
\newcommand{\SL}{\mathrm{SL}}

\newcommand{\A}{\mathbb{A}}
\newcommand{\La}{\mathrm{L}}

\newcommand{\diag}{\mathrm{diag}}
\renewcommand{\l}{\lambda}
\newcommand{\C}{\mathbb{C}}
\newcommand{\Q}{\mathbb{Q}}
\newcommand{\Wh}{\mathcal{W}}
\newcommand{\Ki}{\mathcal{K}}
\newcommand{\N}{\mathbb{N}}
\newcommand{\M}{\mathcal{M}}
\newcommand{\F}{\mathbb{F}}
\newcommand{\SAF}{\mathcal{SAF}}

\newcommand{\Hom}{\mathrm{Hom}}
\newcommand{\Gal}{\mathrm{Gal}}
\newcommand{\Res}{\mathrm{Res}}
\newcommand{\Rep}{\mathrm{Rep}}
\newcommand{\Sp}{\mathrm{Sp}}

\newcommand{\LQ}{\mathrm{LQ}}

\newcommand{\Ind}{\mathrm{Ind}}

\newcommand{\Z}{\mathbb{Z}}

\begin{document}
\title[Distinction inside L-packets of $\SL(n)$]
{Distinction inside L-packets of $\SL(n)$}
\author{U. K. Anandavardhanan and Nadir Matringe}

\address{Department of Mathematics, Indian Institute of Technology Bombay, Mumbai - 400076, India.}
\email{anand@math.iitb.ac.in}

\address{Laboratoire Math\'ematiques et Applications, Universit\'e de Poitiers, France.}
\email{matringe@math.univ-poitiers.fr}


\date{}

\begin{abstract}
If $E/F$ is a quadratic extension $p$-adic fields, we first prove that the $\SL_n(F)$-distinguished representations inside a distinguished unitary $\La$-packet of $\SL_n(E)$ are precisely those admitting a degenerate Whittaker model with respect to a degenerate character of $N(E)/N(F)$. Then we establish a global analogue of this result. For this, let $E/F$ be a quadratic extension of number fields and let $\pi$ be an $\SL_n(\A_F)$-distinguished square integrable automorphic representation of $\SL_n(\A_E)$. Let $(\sigma,d)$ be the unique pair associated to $\pi$, where $\sigma$ is a cuspidal representation of $\GL_r(\A_E)$ with $n=dr$. Using an unfolding argument, we prove that an element of the $\La$-packet of $\pi$ is distinguished with respect to $\SL_n(\A_F)$ if and only if it has a degenerate Whittaker model for a degenerate character $\psi$ of type $r^d:=(r,\dots,r)$ of $N_n(\A_E)$ which is trivial on $N_n(E+\A_F)$, where $N_n$ is the group of unipotent upper triangular matrices of $\SL_n$. As a first application, under the assumptions that $E/F$ splits at infinity and $r$ is odd, we establish a local-global principle for $\SL_n(\A_F)$-distinction inside the $\La$-packet of $\pi$. As a second application we construct examples of distinguished cuspidal automorphic representations $\pi$ of $\SL_n(\A_E)$ such that the period integral vanishes on some canonical copy of $\pi$, and of  everywhere locally distinguished representations of $\SL_n(\A_E)$ such that their $\La$-packets do not contain any distinguished representation.
\end{abstract}

\maketitle

\section{Introduction}\label{intro}

The present work fits in the study of local distinction and periods of automorphic forms, with respect to Galois pairs of reductive groups. It is motivated by earlier works, namely \cite{ap03,ap18} in the local context, and \cite{ap06,ap13} in the global context, which investigated distinction in the presence of $\La$-packets.

In probing distinction inside an $\La$-packet for $\SL(2)$, the key finding of \cite{ap03,ap06} was that distinction inside an $\La$-packet that contains at least one distinguished representation can be characterized in terms of Whittaker models; i.e., distinguished representations in such ``distinguished" $\La$-packets are precisely the ones which admit a Whittaker model with respect to a non-trivial character of $E/F$ (resp. $\A_E/(E+\A_F)$) in the local (resp. global) case. A crucial role in the global papers on $\SL(2)$ \cite{ap06,ap13} is played by ``multiplicity one for $\SL(2)$", i.e., a cuspidal representation of $\SL_2(\A_L)$ appears exactly once in the space of cusp forms on $\SL_2(\A_L)$ \cite{ram00}.            

More recently, the results of \cite{ap18} generalized \cite{ap03} from $n=2$ to any $n$. Thus, in \cite{ap18}, it is proved, amongst many other results, that if $\pi$ is a generic $\SL_n(F)$-distinguished representation of $\SL_n(E)$, then the distinguished members of the L-packet of $\pi$ are the representations which are $\psi$-generic with respect to some non-degenerate character $\psi$ satisfying $\psi^\theta=\psi^{-1}$, where $\theta$ denotes the Galois involution. Such a relationship between distinction and genericity is expected more generally \cite{pra15}; indeed, if $\psi$ is a non-degenerate character such that $\psi^\theta=\psi^{-1}$, according to \cite[Conjecture 13.3, (3)]{pra15}, for any quasi-split Galois pair, $\psi$-generic members of a distinguished $\La$-packet are distinguished. 

Somewhat surprisingly even the finite field analogue of this characterization of distinction in a generic $\La$-packet turned out to be  non-trivial and is settled only fairly recently \cite[Theorem 5.1]{am18}.

In this paper, we first prove a generalization of the above mentioned local result of \cite{ap03,ap18} for unitary L-packets of $\SL_n(E)$ and degenerate Whittaker models (cf. Theorem \ref{theorem main p-adic}).

\begin{theorem}\label{theorem0}
If $\widetilde{\pi}$ is an irreducible unitary representation of $\GL_{n}(E)$ of type $(n_1,\dots,n_d)$, where $(n_1,\dots,n_d)$ is the partition of $n$ defined in Section \ref{type}, and if the L-packet associated to $\widetilde{\pi}$ contains a representation distinguished by $\SL_n(F)$, then its distinguished members are those which admit a $\psi$-degenerate Whittaker model for $\psi$ of type $(n_1,\dots,n_d)$ satisfying $\psi^\theta=\psi^{-1}$.
\end{theorem}

Our proof of Theorem \ref{theorem0} builds on the work of the second named author \cite{mat14}, which classified unitary representations of $\GL_n(E)$ which are distinguished with respect to $\GL_n(F)$, making use of which we can adapt the techniques of \cite{ap03} and \cite{ap18} to the unitary context. Such a result hints at the possibility of a generalization of Prasad's prediction \cite{pra15}, relating distinction for Galois pairs inside distinguished generic $\La$-packets to distinguished Whittaker models, to non-generic $\La$-packets.

Now we come to the global results of this paper. The study of global representations of $\SL(n)$, already quite involved for $n=2$ as can be seen from \cite{ap06,ap13}, is considerably more difficult for several reasons one of which is that ``multiplicity one" is not true for $\SL(n)$ for $n \geq 3$ as was first shown in the famous work of D. Blasius \cite{bla94,lap99}. 

In this paper, we prove the most basic result about characterizing distinction inside a distinguished $\La$-packet in terms of Whittaker models, thus generalizing \cite[Theorem 4.2]{ap06} from $n=2$ to any $n$, and we cover not just cuspidal representations but the full residual spectrum (cf. Theorem \ref{theorem main}).

\begin{theorem}\label{theorem1}
Suppose $\widetilde{\pi}=\Sp(d,\sigma)$ is an irreducible square-integrable automorphic representation of $\GL_{dr}(\A_E)$, where $\sigma$ is a cuspidal representation of $\GL_r(\A_E)$. Assume that the $\La$-packet determined by $\widetilde{\pi}$ contains an $\SL_{dr}(\A_F)$-distinguished representation. Then an irreducible square-integrable automorphic representation $\pi$ of this $\La$-packet is $\SL_{dr}(\A_F)$-distinguished if and only if there exists a degenerate character $\psi$ of type $r^d:=(r,\dots,r)$ (see Section \ref{section global degenerate whittaker models}) of $N_n(\A_E)$ trivial on $N_n(E+\A_F)$ such that $\pi$ has a degenerate $\psi$-Whittaker model.
\end{theorem}

There are two main ideas in proving Theorem \ref{theorem1}. First we settle the cuspidal case by creating an inductive set up based on an unfolding method, and make use of the base case for $n=2$ which is known by \cite[Theorem 4.2]{ap06}. We mention here that the method that we follow to create this inductive set up is very parallel to that employed in \cite[Section 5]{dp19} (cf. Remark \ref{dp-tams}). Having established the cuspidal case for all $r$, we do one more induction, this time in $d$, where $n=dr$, the case $d=1$ being the cuspidal case. In order to work this out, the key ingredient is the work of Yamana \cite{yam15}, which is the global counterpart of \cite{mat14}, and we need to do one more unfolding argument as well.

As an application of Theorem \ref{theorem1}, we establish a local global principle for square-integrable representations for $(\SL_n(\A_E),\SL_n(\A_F))$ (cf. Theorem \ref{theorem localglobal}).

\begin{theorem}\label{theorem2}
Let $E/F$ be a quadratic extension of number fields split at the archimedean places. Suppose $\widetilde{\pi}=\Sp(d,\sigma)$ is a square-integrable automorphic representation of $\GL_{dr}(\A_E)$, where $\sigma$ is a cuspidal representation of $\GL_r(\A_E)$, where we assume that $r$ is odd. Suppose that the $\La$-packet determined by $\widetilde{\pi}$ contains an $\SL_{dr}(\A_F)$-distinguished representation. Let $\pi$ be an irreducible square-integrable automorphic representation of $\SL_n(\A_E)$ which belongs to this $\La$-packet. Write $\pi =\otimes'_v \pi_v$, for $v$ varying through the places of $F$. Then, $\pi$ is distinguished with respect to $\SL_n(\A_F)$ if and only if each $\pi_v$ is $\SL_n(F_v)$-distinguished. 
\end{theorem}

\begin{remark}
Such a local global principle was proved in \cite{ap06} for cuspidal representations of $\SL_2(\A_E)$ by quite involved arguments. In contrast, our proof is reasonably elementary, making use of the assumption that $r$ is odd.
\end{remark}

Another important objective of the present paper is to analyse distinction vis-a-vis the phenomenon of higher multiplicity for $\SL(n)$. As mentioned earlier, unlike in the case of $\SL(2)$, a cuspidal representation may appear in the space of cusp forms with multiplicity more than $1$ for $\SL(n)$ for $n \geq 3$ \cite{bla94,lap99}.  

In our first set of examples, we give a precise answer regarding the non-vanishing of the period integral on the canonical realizations of a cuspidal representation inside the $\La$-packets obtained from restricting the cusp forms on $\GL_n(\A_E)$. We exhibit two types of examples of cuspidal representations of $\SL_n(\A_E)$ of multiplicity $m(\pi)$ more than $1$ in the space of cusp forms which are $\SL_n(\A_F)$-distinguished (cf. \S \ref{bp}, \S \ref{examples}). In one set of examples, $F$ is any number field and $E/F$ is chosen so that the period integral vanishes on some of the $m(\pi)$ many canonical realizations but not on all the canonical realizations. In the second set of examples, $F$ is any number field and $E/F$ is chosen so that the period integral does not vanish in any of the $m(\pi)$ many canonical realizations inside the $\La$-packets. 

Then we tweak the method employed to construct the above examples to also show that the local-global principle fails at the level of non-distinguished $\La$-packets for $\SL(n)$ (cf. \S \ref{subsection examples 3}). Namely we give examples of cuspidal representations $\pi$ of $\SL_n(\A_E)$ which are distinguished at every place, but such that the L-packet of $\pi$ contains no distinguished representation. Such a phenomenon was observed for $\SL(2)$ as well by an explicit construction in \cite[Theorem 8.2]{ap06}. The construction in \cite{ap06} is somewhat involved whereas our analogous examples in \S \ref{subsection examples 3} are conceptually simpler, however the methods here are tailor-made for $n$ odd.

All our examples of cuspidal representations of $\SL_n(\A_E)$ of high multiplicity which are $\SL_n(\A_F)$-distinguished in Section \ref{questions}, that highlight a variety of different phenomena, owe a lot to the examples of Blasius of high cuspidal multiplicity \cite{bla94}. Blasius makes use of the representation theory of the Heisenberg group $H$, and in particular the fact that different Heisenberg representations are such that their value at any element of the group are conjugate in ${\rm PGL}_n(\C)$ but they are projectively inequivalent \cite[Section 1.1]{bla94}. To give a rough idea, Blasius produces in \cite{bla94} high multiplicity examples on $\SL_n(\A_E)$ by transferring this representation theoretic information about Heisenberg groups to Galois groups of $L/E$ for suitable number fields, via Shafarevich's theorem, and then to the automorphic side via the strong Artin conjecture which is a theorem in the situation at hand, as Gal$(L/E) \simeq H$ is nilpotent, due to Arthur-Clozel \cite[Theorem 7.1]{ac89}. For our examples, we start with an involution on $H$ and consider the corresponding semi-direct product $H \rtimes \Z/2$, which cuts out extensions $L \supset E \supset F$, and play with these involutions to construct a variety of examples answering several natural questions about distinction for the pair $(\SL_n(\A_E),\SL_n(\A_F))$.

Finally we mention that we give proofs of some elementary, and probably standard, facts on archimedean and global L-packets of $\SL_n$ for which we could not find accessible sources in the literature. They follow from \cite{ags15} in the archimedean setting, and from \cite{jl13} in the global setting. 

\section{Notations}

We denote by $\d_G$ the character of a locally compact group $G$ such that $\d_G\lambda$ is a right invariant Haar 
measure on $G$ if $\l$ is a left invariant Haar measure on $G$. 
We denote by $\M_{a,b}$ the algebraic group of $a\times b$ matrices. We denote by $G_n$ the algebraic group $\GL_n$, 
by $T_n$ its diagonal torus and by $N_n$ the group of upper triangular matrices in $G_n$. 
We set 
\[U_n=\{u_n(x)=\begin{pmatrix}I_{n-1} & x \\ & 1 \end{pmatrix},\ x\in (\A^1)^{n-1}\}<N_n\] where $\A^1$ denotes the affine line. For $k\leq n$, we embed 
$G_k$ inside $G_n$ via $g\mapsto \diag(g,I_{n-k})$ and set $P_n=G_{n-1}U_n$ the mirabolic subgroup of $G_n$. We denote by $N_{n,r}$ the group of matrices 
\[k(a,x,u)=\begin{pmatrix} a & x \\ & u\end{pmatrix}\]
with $a\in G_{n-r}$, 
$x\in \M_{n-r,r}$ and $u\in N_r$. We denote by $U_{n,r}$ the unipotent radical of $N_{n,r}$, which consists of the matrices $k(I_{n-k}, x, u)$. 
Note that $N_{n,n}=N_n$ and
\[U_{n,r} = U_n\dots U_{r+1}.\] For a subgroup $H$ of $G_n$, we denote by $H^\circ$ the intersection of $H$ with $\SL_n$.

\section{Non-archimedean theory}\label{section non-archimedean theory}

Let $E/F$ be a quadratic extension of $p$-adic fields with Galois involution $\theta$. We denote by 
$|\ |_E$ and $|\ |_F$ the respective normalized absolute values. 
In this section, by abuse of notation, we set $G=G(E)$ for any algebraic group defined over $E$. We denote by $\nu_E$ (or $\nu$), the character 
$|\ |_E\circ \det$ of $G_n$. We fix a non-trivial character $\psi_{0}$ of $E$ which is trivial on $F$.

\subsection{The type of an irreducible  $\GL$-representation via derivatives}\label{type}

If $\psi$ is a non-degenerate (smooth complex) character of $N_n$, we denote by $\psi^k$ its restriction to $U_k$ for $k\leq n$. 
We denote by $\Rep(\bullet)$ the category of smooth complex representations of 
$\bullet$. In \cite{bz76} and \cite{bz77}, Bernstein and Zelevinsky have introduced the functors 
\[{\Phi}_{\psi^n}^-:\Rep(P_n)\rightarrow \Rep(P_{n-1})\] and
 \[\Psi^-:\Rep(P_n)\rightarrow \Rep(G_{n-1}).\] For 
$(\tau,V)\in \Rep(P_n)$, one has 
\[{\Phi}_{\psi^n}^-(V)=V/V(U_n,\psi_n)\] where $V(U_n,\psi_n)$ is the space spanned by the differences $\tau(u)v-\psi_n(u)v$ for $u\in U_n$ and 
$v\in V$, but the action of $P_{n-1}$ on ${\Phi}_{\psi^n}^-(V)$ is normalized by twisting by $\d_{P_n}^{-1/2}$. Similarly 
\[{\Psi}^-(V)=V/V(U_n,1)\] where the action of $G_{n-1}$ on ${\Psi}^-(V)$ is normalized by twisting by $\d_{P_n}^{-1/2}$ again. 

The functor 
${\Phi}_{\psi^n}^-$ does not in fact depend on $\psi$ in the sense that for $\tau\in \Rep(P_n)$ one has 
${\Phi}_{\psi^n}^-(\tau)\simeq {\Phi}_{{\psi'}^n}^-(\tau)$ whenever $\psi$ and $\psi'$ are non-degenerate characters of $N_n$. Hence we simply write 
$\Phi^-(\tau)$ for it. For $\tau\in \Rep(P_n)$, we set 
\[\tau_{(k)}=(\Phi^{-})^{k-1}(\tau)\in \Rep(P_{n+1-k}),\]  and
\[\tau^{(k)}=\Psi^-(\Phi^{-})^{k-1}(\tau)\in \Rep(G_{n-k}),\] which is called the $k$-th derivative of $\tau$. 
The $k$-th shifted derivative of $\tau$ is given by
\[\tau^{[k]}=\nu^{1/2}\tau^{(k)}.\] 
Note that these definitions apply when $\tau$ is a representation of $G_n$ which we consider as a representation of $P_n$ by restriction. 

Let $\widetilde{\pi}$ be an irreducible smooth representation of $G_n$. We denote by $\widetilde{\pi}^{[n_1]}$ its highest (non-zero) 
shifted derivative, by $\widetilde{\pi}^{[n_1,n_2]}:=(\widetilde{\pi}^{[n_1]})^{[n_2]}$ the highest shifted derivative of
$\widetilde{\pi}^{[n_1]}$, and so on. All the representations $\widetilde{\pi}^{[n_1,n_2,\dots,n_i]}$ are irreducible thanks to
\cite[Theorem 8.1]{zel80}. This defines a finite sequence of positive integers $(n_1,\dots,n_d)$ such than $n_1+\dots+n_d=n$. In 
fact, \cite[Theorem 8.1]{zel80} implies that this sequence is a partition of $n$, i.e., $n_1\geq n_2 \geq \dots \geq n_d$. We call 
$(n_1,\dots,n_d)$ \textit{the partition associated to $\widetilde{\pi}$}. We will also say that $\widetilde{\pi}$ \textit{is of type} $(n_1,\dots,n_d)$. 
Note that by \cite[Section 7.4]{ber84}, if $\widetilde{\pi}$ is unitary, then all the 
representations $\widetilde{\pi}^{[n_1,n_2,\dots,n_i]}$ are unitary as well.

\begin{example}
Using the product notation for normalized parabolic induction, if $\d$ is an essentially square-integrable representation of $G_r$ we set 
 \[\Sp(d,\d)=\LQ(|\ |_E^{(d-1)/2}\d \times \dots \times |\ |_E^{(1-d)/2}\d)\] to be the Langlands quotient of the parabolically induced representation 
 \[|\ |_E^{(d-1)/2}\d \times \dots \times |\ |_E^{(1-d)/2}\d.\] 
 More generally, if 
 $\tau=\d_1\times \dots \times \d_l$ is a generic unitary representation of $G_r$ written as a commutative product of essentially square-integrable representations (\cite[Theorem 9.7]{zel80}), we set 
 \[\Sp(d,\tau)=\Sp(d,\d_1)\times \dots \times \Sp(d,\d_l)\] which is a commutative product by the results of Tadic (\cite[Theorem D]{tad86}).
 In this situation, \cite[4.5, Lemma]{bz77} together with the computation of the highest derivative of Speh representations (\cite[\S 3.5 (3.3)]{os08} and \cite[\S 6.1]{tad87}) imply that the partition of $n=rd$ associated 
 to $\Sp(d,\tau)$ is $r^d:=(r,\dots,r)$. Conversely one can check using the same results that an irreducible unitary representation of $G_n$ of type $r^d$ is of the form $\Sp(d,\tau)$ for a unitary generic representation $\tau$ of $G_r$. We refer to Section \ref{subsection degenerate archimedean whittaker models} for the details in the archimedean setting which are the same as in the non-archimedean setting.
\end{example}

\subsection{Degenerate Whittaker models and $\La$-packets}\label{section p-adic nd Whittaker models}

Let $\psi_{n_i}$ be a non-degenerate character of the group $N_{n_i}$. By \cite[Section 8]{zel80}, if the representation 
$\widetilde{\pi}$ is of type $(n_1,\dots,n_d)$, then it has a unique degenerate Whittaker model with respect to \[(\psi_{n_1}\otimes \dots \otimes \psi_{n_d})\begin{pmatrix} u_d & \dots & . \\
 &  \ddots  & \vdots \\ & & u_1 \end{pmatrix}=\psi_{n_1}(u_1)\dots \psi_{n_d}(u_d)\] for $u_i\in N_{n_i}$. We often use the notation 
 \[\psi_{n_1,\dots,n_d}:=\psi_{n_1}\otimes \dots \otimes \psi_{n_d},\] it has 
 the advantage of being short but could mislead the reader so we insist on the fact that $\psi_{n_1,\dots,n_d}$ 
 depends on the characters $\psi_{n_i}$ and not only on the positive integers $n_i$. 
  We will say that 
 $\psi_{n_1,\dots,n_d}$ \textit{is of type} $(n_1,\dots,n_d)$. If all the 
 $n_i$ are equal then we set 
 \[\psi_{1,\dots,d}:=\psi_{n_1,\dots,n_d}.\]
 
The $\La$-packet associated to $\widetilde{\pi}$ is the finite set of irreducible representations of $G_n^\circ=\SL_n(E)$ appearing in 
 the restriction of $\widetilde{\pi}$, and is denoted by $\La(\widetilde{\pi})$. We refer to \cite[Section 2]{hs12} for its basic properties that we now state (see also \cite{gk82} or \cite{tad92}). 
 Any irreducible representation $\pi$ of $G_n^\circ$ arises in the restriction of an irreducible representation of $G_n$ and two 
 irreducible representations of $G_n$ containing $\pi$ are 
 twists of each other by a character. Hence it makes sense to set $\La(\pi)=\La(\widetilde{\pi})$, and call this finite set the 
 $\La$-packet of $\pi$ (or the L-packet determined by $\widetilde{\pi}$).
 \textit{We define the type of $\pi$ (or the type of $\La(\pi)$) to be that of $\widetilde{\pi}$.} Of course two irreducible representations of $G_n$ determining the same $\La$-packet have the same type. 
 
Clearly the group $\diag(E^\times,I_{n-1})$ acts transitively on $\La(\widetilde{\pi})$ and the existence 
 of a degenerate Whittaker model for irreducible representations of $G_n$ then has the following immediate consequence.
 
\begin{lemma}\label{lemma existence of nd Whitt model for p-adic L-packets}
Let $\widetilde{\pi}$ be an irreducible representation of $G_n$ of type $(n_1,\dots,n_d)$, then the group $\diag(E^\times,I_{n-1})$ acts transitively on $\La(\widetilde{\pi})$ and every member of $\La(\widetilde{\pi})$ has a (necessarily unique) degenerate $\psi$-Whittaker model for some $\psi$ of type $(n_1,\dots,n_d)$.
\end{lemma}
 
Uniqueness of degenerate Whittaker models for $\widetilde{\pi}$ together with Lemma \ref{lemma existence of nd Whitt model for p-adic L-packets} then has the following well-known consequence.
 
\begin{proposition}\label{proposition multiplicity one inside non arch L-packets}
If $\widetilde{\pi}$ is an irreducible representation of $G_n$ then the representations in $\La(\widetilde{\pi})$ appear with multiplicity one 
in the restriction of $\widetilde{\pi}$ to $G_n^\circ$. 
\end{proposition} 

In fact we can be more precise. The following lemma follows from the fact that if $\widetilde{\pi}$ is of 
type $(n_1,\dots,n_d)$ then $\widetilde{\pi}^{[n_1,\dots,n_{k-1}]}$ is of type $(n_k,\dots,n_d)$ (cf. \S \ref{type}). 

\begin{lemma}\label{lemma irrep with fixed wh model in the derivative}
If $\widetilde{\pi}$ is an irreducible representation of $G_n$ of type $(n_1,\dots,n_d)$, then $\La(\widetilde{\pi}^{[n_1,\dots,n_{k-1}]})$ contains a unique irreducible representation of $G_{n_k+\dots+n_d}^\circ$ with a degenerate Whittaker model with respect to 
$\psi_{n_k,\dots,n_d}$. 
\end{lemma}

Again $\La(\widetilde{\pi}^{[n_1,\dots,n_{k-1}]})$ only depends on $\La(\widetilde{\pi})=\La(\pi)$ (because derivatives commute with character twists), and we set 
\[\La(\pi)^{[n_1,\dots,n_{k-1}]}:=\La(\widetilde{\pi}^{[n_1,\dots,n_{k-1}]})\] for any irreducible representation $\widetilde{\pi}$ of 
$G_n$ such that $\pi\in \La(\widetilde{\pi})$.

\begin{definition}\label{definition irrep with fixed wh model in the derivative}
Let $\pi$ be an irreducible representation of $G_n^\circ$. Let $\pi^{[n_1,\dots,n_{k-1}]}(\psi_{n_k,\dots,n_d})$ denote the irreducible representation of $G_{n_k+\dots+n_d}^\circ$ isolated in Lemma 
\ref{lemma irrep with fixed wh model in the derivative}, i.e., the unique representation in $\La(\pi)^{[n_1,\dots,n_{k-1}]}$ with a degenerate Whittaker model with respect to $\psi_{n_k,\dots,n_d}$. In particular, $\pi(\psi_{n_1,\dots,n_d})$ denotes the unique 
irreducible representation of $G_n^\circ$ in $\La(\pi)$ with a degenerate Whittaker model with respect to $\psi_{n_1,\dots,n_d}$.
\end{definition}

\begin{remark}
We do not claim that if $\pi(\psi)=\pi(\psi')$, then $\psi$ and $\psi'$ are in the same $T_n^\circ$-conjugacy class.
\end{remark}

\subsection{Distinguished representations inside a distinguished $\La$-packet}\label{section distinction inside p-adic packets}

Let $\widetilde{\pi}$ be an irreducible representation of $G_n$. We start by making explicit the relation between the degenerate Whittaker models $\Wh(\widetilde{\pi},\psi_{n_1,\dots,n_d})$ and 
$\Wh(\widetilde{\pi}^{[n_1]},\psi_{n_{2},\dots,n_d})$. 

\begin{lemma}\label{lemma compatibility restriction whittaker model}
The map \[W\mapsto  W_{|G_{n-n_1}}\] is surjective from 
$\Wh(\widetilde{\pi},\psi_{n_1,\dots,n_d})$ to $\Wh(\nu_E^{(n_1-1)/2}\widetilde{\pi}^{[n_1]},\psi_{n_2,\dots,n_d})$. 
\end{lemma}
\begin{proof}
By the same proof as in \cite[Proposition 1.2]{cps17}, the map 
\[W\mapsto W_{|P_{n-n_1+1}}\]
is a surjection from 
$\Wh(\widetilde{\pi},\psi_{n_1,\dots,n_d})$ to $\Wh(\nu_E^{(n-n_1+1)/2}\widetilde{\pi}_{(n_1-1)},\psi_{n_2,\dots,n_d})$. But then, because 
$ W(gu)=W(g)$ for $W\in \Wh(\widetilde{\pi},\psi_{n_1,\dots,n_d})$, $g \in G_{n-n_1}$ and $u\in U_{n-n_1+1}$, we deduce that 
$W_{|G_{n-n_1}}\in \Wh(\nu_E^{n_1/2}\widetilde{\pi}^{(n_1)},\psi_{n_2,\dots,n_d})$ and that 
\[W\mapsto W_{|G_{n-n_1}}\]
is surjective from $\Wh(\widetilde{\pi},\psi_{n_1,\dots,n_d})$ to $\Wh(\nu_E^{n_1/2}\widetilde{\pi}^{(n_1)},\psi_{n_2,\dots,n_d})$. The result follows.
\end{proof}

We denote by $\Ki(\widetilde{\pi},\widetilde{\pi}^{[n_1]},\psi_{n_1})$, the generalized Kirillov model of $\widetilde{\pi}$ (see \cite[Section 5]{zel80}) with respect to $\widetilde{\pi}^{[n_1]}$ and $\psi_{n_1}$. It is, by definition, the image of the unique embedding of $\widetilde{\pi}_{|P_n}$ into the space of functions $K:P_n\rightarrow \pi^{[n_1]}$ which satisfy \[K(k(a,x,u_1)p)=\nu(a)^{(n_1-1)/2}\psi_{n_1}(u_1)\pi^{[n_1]}(a)K(p)\] for $k(a,x,u_1)\in N_{n,n_1}$.

Let $\widetilde{\pi}$ be an irreducible representation of $G_n$ with degenerate Whittaker model 
$\Wh(\widetilde{\pi},\psi_{n_1,\dots,n_d})$. Then, by Lemma \ref{lemma compatibility restriction whittaker model}, for any 
$W\in \Wh(\widetilde{\pi},\psi_{n_1,\dots,n_d})$ and $g\in G_n$, the map 
\[g_1\mapsto \nu_E^{(n_1-1)/2} W(g_1g)\]
belongs to 
$\Wh(\widetilde{\pi}^{[n_1]},\psi_{n_2,\dots,n_d})$. We set 
\[I(W): G_n \rightarrow \Wh(\nu_E^{(n_1-1)/2}\widetilde{\pi}^{[n_1]},\psi_{n_2,\dots,n_d})\]
to be the map defined by \[I(W)(g): g_1\in G_{n-n_1}\mapsto  W(g_1g).\] 
Hence $I$ realizes 
$\Wh(\widetilde{\pi},\psi_{n_1,n_2,\dots,n_d})$ inside the induced representation
\[\Ind_{N_{n,n_1}}^{G_n}(\Wh(\nu_E^{(n_1-1)/2}\widetilde{\pi}^{[n_1]},\psi_{n_2,\dots,n_d})\otimes \psi_{n_1}).\] 
Then the map 
$W\mapsto I(W)_{|P_n}$ is a bijection from 
\[\Wh(\widetilde{\pi},\psi_{n_1,n_2,\dots,n_d}) \rightarrow \Ki(\widetilde{\pi}, \Wh(\widetilde{\pi}^{[n_1]},\psi_{n_2,\dots,n_d}),\psi_{n_1})).\]

The following is now a consequence of the results of \cite{mat14}.

\begin{proposition}\label{proposition mat14}
Let $\widetilde{\pi}$ be an irreducible unitary representation of $G_n$ which is distinguished with respect to $G_n^\theta$, with degenerate Whittaker model 
$\Wh(\widetilde{\pi},\psi_{n_1,\dots,n_d})$, and suppose that $\psi_{n_1,\dots,n_d}$ is trivial on $N_n(F)$. Then the 
invariant linear form on $\widetilde{\pi}$ is expressed as a local period on $\Wh(\widetilde{\pi},\psi_{n_1,\dots,n_d})$ by
\[\lambda(W)= \int_{N_{n,n_1}^\theta\backslash P_n^\theta}
\int_{N_{n-n_1,n_2}^\theta\backslash P_{n-n_1}^\theta} \dots \int_{N_{n-\sum_{i=1}^{d-1} n_i,n_d}^\theta\backslash P_{n-\sum_{i=1}^{d-1}n_i}^\theta} W(p_d\dots p_2 p_1) d p_d \dots d p_2 d p_1.\] 
\end{proposition}
\begin{proof}
The proof is by induction on $d$. For $d=1$, the representation is unitary generic, and the fact that \[W\mapsto \int_{N_n^\theta\backslash P_n^\theta} W(p)dp\]
is well-defined is due to Flicker \cite[Section 4]{fli88}, and that it is $G_n^\theta$-invariant is a result due to Youngbin Ok (see \cite[Proposition 2.5]{mat14} for a more general statement in the unitary context). Then for a general $d$, by \cite[Proposition 2.4]{mat14}, if $\widetilde{\pi}$ is distinguished, so is $\widetilde{\pi}^{[n_1]}$, and we take $L\in \Hom_{G_{n-n_1}^\theta}(\widetilde{\pi}^{[n_1]},\C)-\{0\}$. By \cite[Propositions 2.2 and 2.5]{mat14}, the linear form 
\begin{equation} \label{lK} \l_K:K\mapsto \int_{N_{n,n_1}^{\theta}\backslash P_n^{\theta}} L(K(p_1)) dp_1\end{equation} is, up to scaling, the unique $G_n^\theta$-invariant linear form on $\Ki(\widetilde{\pi},\widetilde{\pi}^{[n_1]},\psi_{n_1})$. We realize $\widetilde{\pi}$ as 
$\Wh(\widetilde{\pi},\psi_{n_1,\dots,n_d})$ and $\widetilde{\pi}^{[n_1]}$ as $\Wh(\widetilde{\pi}^{[n_1]},\psi_{n_2,\dots,n_d})$. Then by induction 
\[L(W')= \int_{N_{n-n_1,n_2}^\theta\backslash P_{n-n_1}^\theta} \dots \int_{N_{n-\sum_{i=1}^{d-1} n_i,n_d}^\theta\backslash P_{n-\sum_{i=1}^{d-1} n_i}^\theta} W'(p_r\dots p_2) d p_r \dots d p_2.\] Applying it to $W'=K(p_1)=I(W_K)(p_1)$ for the unique $W_K\in \Wh(\widetilde{\pi},\psi_{n_1,\dots,n_d})$ such that the previous equality holds, the result follows in view of the discussion preceding the proposition.
\end{proof}
 
\begin{theorem}\label{theorem main p-adic}
Let $\pi$ be an irreducible unitary representation of $\SL_n(E)$ of type $(n_1,\dots,n_d)$ which is $\SL_n(F)$-distinguished. 
Then the $\SL_n(F)$-distinguished representations in $\La(\pi)$ are precisely the representations $\pi(\psi)$ for 
a character $\psi$ of $N_n$ of type $(n_1,\dots,n_d)$ such that 
$\psi_{|N_n^\theta}\equiv \mathbf{1}$.
\end{theorem}
\begin{proof}
The proof follows exactly along the same lines of the generic case as in \cite[Section 3]{ap03} and \cite[Section 4]{ap18}, making use of Proposition \ref{proposition mat14} in lieu of Flicker's invariant linear form mentioned above.
\end{proof}

Theorem \ref{theorem main p-adic} has the following consequences.

\begin{proposition}\label{proposition dist rep with given wh inside derivative}
Let $\pi$ be an irreducible unitary representation of $G_n^\circ$ of type $(n_1,\dots,n_d)$, and fix 
$\psi_{n_1,\dots,n_d}$ a character of $N_n$ of this type trivial on $N_n^\theta$. If $\pi(\psi_{n_1,\dots,n_d})$ is $\SL_n(F)$-distinguished, then the representation 
$\pi^{[n_1,\dots,n_{k-1}]}(\psi_{n_k,\dots,n_d})$ is $\SL_{\sum_{i=k}^d n_i}(F)$-distinguished for all 
$k=1,\dots,d$.
\end{proposition}
\begin{proof}
According to \cite[Lemma 3.2]{ap18}, up to twisting $\widetilde{\pi}$ by an appropriate character we can suppose that it is $\GL_n(F)$-distinguished. Then $\widetilde{\pi}^{[n_1,\dots,n_{k-1}]}$ is distinguished as we already saw (cf. proof of Proposition \ref{proposition mat14}). Now $\pi^{[n_1,\dots,n_{k-1}]}(\psi_{n_k,\dots,n_d})$ belongs to 
$\La(\pi^{[n_1,\dots,n_{k-1}]})$ and it has a degenerate Whittaker model with respect to the distinguished character $\psi_{n_k,\dots,n_d}$, hence the result follows from 
Theorem \ref{theorem main p-adic}.
\end{proof}

Proposition \ref{proposition dist rep with given wh inside derivative} can be strengthened for Speh representations.

\begin{theorem}\label{theorem dist rep with given wh inside generic of Speh}
Let $\tau$ be a generic representation of $G_r$ and let $\psi_i$ be a non-degenerate character of $N_{r}$ trivial 
on $N_r^\theta$ for $i=1,\dots,d$. Fix  
$1\leq k \leq d$, then $\pi(\psi_{1,\dots,d})\in \La(\Sp(d,\tau))$ is $\SL_n(F)$-distinguished if and only if 
$\pi^{[r^{d-k}]}(\psi_{d-k+1,\dots,d})\in \La(\Sp(k,\tau))$ is $\SL_{kr}(F)$-distinguished. 
\end{theorem}
\begin{proof}
One direction follows from Proposition \ref{proposition dist rep with given wh inside derivative}. Conversely suppose that
\[\pi^{[r^{d-k}]}(\psi_{d-k+1,\dots,d})\in \La(\Sp(k,\tau))\] is $\SL_{kr}(F)$-distinguished. 
Then up to a twist $\Sp(k,\tau)$, hence $\tau$, and hence $\Sp(d,\tau)$ is distinguished, thanks to \cite[Theorem 2.13]{mat14}. But then because $\pi(\psi_{1,\dots,d})\in \La(\Sp(d,\tau))$ has a 
$\psi_{1,\dots,d}$-degenerate Whittaker model and that $\psi_{1,\dots,d}$ is trivial on $N_n^\theta$, we deduce that 
$\pi(\psi_{1,\dots,d})$ is $\SL_n(F)$-distinguished, thanks to Theorem \ref{theorem main p-adic}.
\end{proof}

We will give the global analogue of this result in Theorem \ref{theorem SL analogue of yam thm 1.2}.

\section{Archimedean prerequisites for the global theory}

Here $E=\C$ or $\R$ and by abuse of notation we write $G=G(E)$ for any algebraic group defined over $E$. We set 
$|a+ib|_{\C}=a^2+b^2$ and denote by $|\ |_{\R}$ the usual absolute value on $\R$. We then denote by $\nu_E$ the character of $G_n$ 
obtained by composing $|\ |_E$ with $\det$. For $G$ a reductive subgroup of $G_n$ we write $\SAF(G)$ for the category of smooth admissible Fréchet representations of $G$ of moderate growth as in \cite{ags15}, in which we work.
We use the same product notation for parabolic induction in $\SAF(G_n)$ as in \cite{ags15}.

We only consider unitary characters of $N_n$. The non-degenerate characters of $N_n$ are of the form 
\[\psi_{\l}:\begin{pmatrix} 1 & z_1 & \dots & \dots  & . \\  & 1 & z_2  &  \dots & . \\ & & \ddots & \ddots & .  \\ & & &  1 & z_{n-1} 
\\  & & & &  1 \end{pmatrix}\mapsto \exp(i \sum_{i=1}^{n-1} \Re(\l_i z_i))\] with $\l_i\in E^*$. Then for a partition $(n_1,\dots,n_r)$ of $n$ and non-degenerate characters $\psi_{n_i}$ of $N_{n_i}$ we define the degenerate character $\psi_{n_1,\dots,n_d}$ of $N_n$ as in Section \ref{section p-adic nd Whittaker models} and we also write $\psi_{1,\dots,d}:=\psi_{n_1,\dots,n_d}$ when all the $n_i$'s are equal. We again say $\psi_{n_1,\dots,n_d}$ is of type $(n_1,\dots,n_d)$ so that the set of characters of a given type forms a single $T_n$-conjugacy class. We call a member of this conjugacy class a 
 degenerate character of type $(n_1,\dots,n_d)$. For a degenerate character $\psi$ of 
 $N_n$ and an irreducible representation $\widetilde{\pi}$ of $G_n$, by a $\psi$-Whittaker functional, we mean a non-zero continuous linear form $L$ from $\widetilde{\pi}$ to 
 $\C$ satisfying \[L(\widetilde{\pi}(n)v)=\psi(n)L(v)\] for $n\in N_n$ and $v\in \widetilde{\pi}$. We will say that $\widetilde{\pi}$ has a unique $\psi$-Whittaker model if the space of $\psi$-Whittaker functionals on the space of $\widetilde{\pi}$ is one-dimensional.

\subsection{The Tadi\'c classification of the unitary dual of $G_n$}

We recall that irreducible square-integrable representations of $G_n$ for $n\geq 1$ exist only when 
$n=1$ if $E=\mathbb C$ and when $n=1$ or $2$ if $E=\mathbb R$. When $n=1$ these are just the unitary characters of $E^\times$. For $d\in \N$ and an irreducible square-integrable representation $\d$ of $G_n$ ($n=1$ or $n \in \{1,2\}$ depending on $E$ is $\mathbb C$ or $\mathbb R$) we denote by \[\Sp(d,\d)=\LQ(\nu_E^{(d-1)/2}\d \times \dots \times \nu_E^{(1-d)/2}\d)\] the Langlands quotient of 
$\nu_E^{(d-1)/2}\d \times \dots \times \nu_E^{(1-d)/2}\d$. In particular, $\Sp(d,\chi)=\chi\circ \det$ when 
$\chi$ is a unitary character of $G_1$. By \cite{tad09}, the 
representations \[\pi(\Sp(d,\d),\alpha):=\nu^\alpha \Sp(d,\d)\times \nu^{-\alpha} \Sp(d,\d)\] are irreducible unitary when $\alpha\in (0,1/2)$, and any irreducible representation $\pi$ of $G_n$ can be written in a unique manner as a commutative product 
\[\widetilde{\pi}=\prod_{i=1}^r\Sp(d_i,\d_i)\prod_{j=r+1}^s \pi(\Sp(d_j,\d_j),\alpha_j).\] 
When all the $d_i$ and $d_j$ are equal to one, the representation 
\[\tau=\prod_{i=1}^r \d_i\prod_{j=r+1}^s \pi(\d_j,\alpha_j)\] is generic unitary (it has a unique $\psi$-Whittaker model for any non-degenerate character $\psi$ of $N_n$), according to \cite[p.4]{jac09}, and we set 
\[\Sp(d,\tau)=\prod_{i=1}^r\Sp(d,\d_i)\prod_{j=r+1}^s \pi(\Sp(d,\d_j),\alpha_j)\] which is thus an irreducible unitary representation. 

We note that according to the proof of \cite[4.1.1]{gs13}, which refers to \cite{vog86} and \cite{ss90}, a Speh representation 
$\Sp(d,\d)$ for $\d$ an irreducible square-integrable representation of $G_2$ is the same thing as the Speh representations of 
Vogan's classification as presented in \cite[4.1.2 (c)]{ags15}. Hence the Vogan classification as stated in 
\cite[4.1.2]{ags15} is immediately related to that of Tadi\'c:

\begin{itemize} 
\item the unitary characters of \cite[4.1.2 (a)]{ags15} are the representations 
of the form $\Sp(d,\chi)$ for $\chi$ a unitary character of $G_1$,
\item the Stein complementary series of \cite[4.1.2 (b)]{ags15} are the representations 
of the form $\pi(\Sp(d,\chi),\alpha)$ for $\chi$ a unitary character of $G_1$, 
\item the Speh representations of \cite[4.1.2 (c)]{ags15} are the representations 
of the form $\Sp(d,\d)$ for $\d$ an irreducible square-integrable representation of $G_2$,
\item the Speh complementary series of \cite[4.1.2 (d)]{ags15} are the representations 
of the form $\pi(\Sp(d,\d),\alpha)$ for $\d$ an irreducible square-integrable representation of $G_2$.
\end{itemize}

The third and fourth cases occur only when $E=\R$.

\subsection{Degenerate Whittaker models of irreducible unitary representations}\label{subsection degenerate archimedean whittaker models}

In this section we recall the results of Aizenbud, Gourevitch, and Sahi on degenerate Whittaker models for $\GL_n(E)$ for $E=\C$ or $\R$. We believe that with the material developed by these authors, together with the real analogue of Ok's result due to Kemarsky \cite{kem15}, the results obtained in \cite{mat14} and Section \ref{section non-archimedean theory} are in reach. However, being inexperienced in such matters, we leave this for experts, and simply recall immediate implications of the results in 
\cite{ags15} that we will need for our global applications. 

In \cite{sah89}, to any irreducible representation $\widetilde{\pi}$ 
of $G_n$, Sahi attached an irreducible representation $A(\widetilde{\pi})$ of $G_{n-n_1}$ for some $0<n_1\leq n$, \textit{the adduced representation of $\widetilde{\pi}$} and proved that it satisfied 
\[A(\prod_{i=1}^r\Sp(d_i,\d_i)\prod_{j=r+1}^s \pi(\Sp(d_j,\d_j),\alpha_j))=\prod_{i=1}^rA(\Sp(d_i,\d_i))\prod_{j=r+1}^s A(\pi(\Sp(d_j,\d_j),\alpha_j))\] 
with respect to the Tadi\'c classification. The adduced representation is the archimedean highest shifted derivative, and from \cite{sah90,gs13,ags15} (see \cite[Section 4]{ags15}) one has 
\begin{equation}\label{equation adduced} A(\prod_{i=1}^r\Sp(d_i,\d_i)\prod_{j=r+1}^s \pi(\Sp(d_j,\d_j),\alpha_j))\end{equation}
\[=\prod_{i=1}^r \Sp(d_i-1,\d_i)\prod_{j=r+1}^s \pi(\Sp(d_j-1,\d_j),\alpha_j).\]
 
One can then take the adduced of the adduced representation of the irreducible unitary representation $\widetilde{\pi}$ and so on, and 
obtain the "depth sequence" $\overline{n}:=(n_1,\dots,n_d)$ attached to $\widetilde{\pi}$, which forms a partition of $n$. 
We call this depth sequence the \textit{type of $\widetilde{\pi}$}. The combination of \cite[Theorem A]{gs13} and \cite[Theorem 4.2.3]{ags15} says:

\begin{theorem}\label{theorem archimedean partition}
Let $\widetilde{\pi}$ be an irreducible unitary representation of $G_n$ of type $(n_1,\dots ,n_d)$, 
and $\psi$ be any character of $N_n$ of type $(n_1,\dots ,n_d)$, then $\widetilde{\pi}$ has a unique degenerate $\psi$-Whittaker model.
\end{theorem}

For an irreducible representation $\pi$ of $\SL_n(E)$, the notion of a degenerate Whittaker model is defined similarly, though this time 
the notion does depend on the $T_n^\circ$-conjugacy class of the degenerate character $\psi$ and not just its type. The L-packet of $\pi$ is defined as in the $p$-adic case, and we refer 
to \cite[End of Section 2]{hs12}. Note that \cite{hs12} deals with Harish-Chandra modules but their results remain valid in the context of $\SAF(G_n)$, thanks to the Casselman-Wallach equivalence of categories (see \cite[Chapter 11]{wal88}). 
 If $\widetilde{\pi}$ is an irreducible unitary representation of $G_n$, it follows from \cite[Theorem A]{gs13} that the type of $\widetilde{\pi}$ depends only on $\La(\widetilde{\pi})$, and \textit{we define the type of an irreducible unitary representation $\pi$ of $G_n^\circ$ to be that of any irreducible representation $\widetilde{\pi}$ of $G_n$ such that 
 $\pi\in \La(\widetilde{\pi})$}.

\begin{remark}
If $\widetilde{\pi}$ is an irreducible representation of $G_n^\circ$, then $\widetilde{\pi}_{|G_n^\circ}$ contains an irreducible unitary representation if and only if it is unitary up to a character twist.
\end{remark} 
 
As in the $p$-adic case, Theorem \ref{theorem archimedean partition} has the following consequence. 

\begin{corollary}\label{corollary existence of deg Whitt model and multiplicity one inside arch L packets}
Let $\widetilde{\pi}$ be a smooth irreducible unitary representation of $G_n$ of type $(n_1,\dots ,n_d)$. Then the group 
$\diag(E^\times,I_{n-1})$ acts transitively on $\La(\widetilde{\pi})$ and every $\pi\in \La(\widetilde{\pi})$ has a (necessarily unique) degenerate $\psi$-Whittaker model for some character $\psi$ of $N_n$ of type $(n_1,\dots ,n_d)$. Moreover, $\widetilde{\pi}_{|G_n^\circ}$ is multiplicity free.
\end{corollary}

We note that the computation of the adduced representation given in Equation (\ref{equation adduced}) implies:

\begin{theorem}[Aizenbud, Gourevitch, and Sahi]\label{theorem ags type of speh reps}
Let $\tau$ be an irreducible generic representation of $G_r$. The Speh representation $\Sp(d,\tau)$ has type $r^d$, and conversely an irreducible unitary representation of $G_n$ of type $r^d$ is of the form $\Sp(d,\tau)$ for some unitary generic representation $\tau$ of $G_r$.
\end{theorem}

We end by giving the archimedean analogue of Definition \ref{definition irrep with fixed wh model in the derivative} for 
Speh representations.

\begin{definition}\label{definition irrep with fixed wh model in the archimedean derivative}
Let $\pi$ be an irreducible unitary representation of $G_n^\circ$ of type $r^d$, and let $\tau$ be an irreducible unitary generic representation of $G_r$ such that  $\pi \in \La(\Sp(d,\tau))$. For $\psi_{1,\dots,d}$ a character of $N_n$ of type $r^d$, 
we denote by 
$\pi^{[r^{d-k}]}(\psi_{d-k+1,\dots,d})$ the unique representation in $\La(\Sp(k,\tau))$ with a $\psi_{d-k+1,\dots,d}$-degenerate Whittaker model. 
\end{definition}

\begin{remark}
The representation $\pi^{[r^{d-k}]}(\psi_{d-k+1,\dots,d})$ above depends only on $\La(\pi)$.
\end{remark}

\section{The Global setting}\label{section global}

In this section, $E/F$ is a quadratic extension of number fields with associated Galois involution $\theta$. We denote by $\A_E$ and $\A_F$ the rings of adeles of $E$ and $F$ respectively. We denote by $\GL_n(\A_E)^1$ the elements of 
$\GL_n(\A_E)$ which have determinant of adelic norm equal to $1$, and for any subgroup $H$ of 
$\GL_n(\A_E)$, by $H^1$ we denote the intersection of $H$ with $\GL_n(\A_E)^1$. We recall that $\A_F^\times=\A_F^1\times (\A_F)_{>0}$, where 
$(\A_F)_{>0}$ is $\R_{>0}\otimes_{\Q} 1\subset \R\otimes _{\Q} F$ sitting inside $\A_F$. In particular passing to 
the groups of unitary characters we have $\widehat{\A_F^\times}=\widehat{\A_F^1}\times \widehat{(\A_F)_{>0}}$, and for 
$\l \in \R$ we denote by $\alpha_{\l}$ the unitary character of $\A_F^\times$ corresponding to 
$(\alpha,(|\ |_{\A_F}^{i\l})_{|\widehat{(\A_F)_{>0}}})\in \widehat{\A_F^1}\times \widehat{(\A_F)_{>0}}$. Namely extending 
$\alpha_0$ is the extension of $\alpha$ which is trivial on $\widehat{(\A_F)_{>0}}$ and $\alpha_{\l}=\alpha_0 |\ |_{\A_F}^{i\l}$. In particular 
$\alpha_\l$ is automorphic if and only if $\alpha\in \widehat{F^\times \backslash \A_F^1}$.

\subsection{Square integrable automorphic representations and their L-packets}

For $\omega\in \widehat{E^\times \backslash \A_E^\times}$, we denote by
\[L^{2}(\A_E^\times \GL_n(E)\backslash \GL_n(\A_E),\omega )\]
the space of 
smooth $L^2$-automorphic forms on which the center $\A_E^\times$ of $\GL_n(\A_E)$ acts by $\omega$, and by 
\[L_d^{2}(\A_E^\times \GL_n(E)\backslash \GL_n(\A_E),\omega)\] its discrete part. We then denote by 
$L_d^{2,\infty}(\A_E^\times \GL_n(E)\backslash \GL_n(\A_E),\omega)$ the dense $\GL_n(\A_E)$-submodule of $L_d^{2}(\A_E^\times \GL_n(E)\backslash \GL_n(\A_E),\omega)$ consisting of smooth automorphic forms (cf. \cite[Lecture 2]{cog04}). We say 
that $\widetilde{\pi}$ is a \textit{square-integrable automorphic representation of $\GL_n(\A_E)$} if it is a closed (for the Fréchet topology) irreducible $\GL_n(\A_E)$-submodule of $L_d^{2,\infty}(\A_E^\times \GL_n(E)\backslash \GL_n(\A_E),\omega)$ for some Hecke character $\omega$. 
The space $L_d^{2,\infty}(\A_E^\times \GL_n(E)\backslash \GL_n(\A_E),\omega)$ contains the space of smooth cusp forms
\[\mathcal{A}^{\infty}_0(\A_E^\times \GL_n(E)\backslash \GL_n(\A_E),\omega)\] 
as a $\GL_n(\A_E)$-invariant subspace. A \textit{cuspidal automorphic representation of $\GL_n(\A_E)$} is a closed irreducible $\GL_n(\A_E)$-submodule of $\mathcal{A}^{\infty}_0(\A_E^\times \GL_n(E)\backslash \GL_n(\A_E),\omega)$, for some Hecke character $\omega$.

Let $\sigma$ be a cuspidal automorphic representation of $\GL_r(\A_E)$, and  
\[\widetilde{\pi}=\Sp(d,\sigma)=\otimes'_v \Sp(d,\sigma_v)\] be the restricted tensor product of the representations $\Sp(d,\sigma_v)$ for $v$ varying through the places of $E$. By \cite{jac84}, this is a square-integrable automorphic representation of $\GL_n(\A_E)$, where $n=dr$. By \cite{mw89}, any irreducible square-integrable automorphic representation of $\GL_n(\A_E)$ is of this form for a unique pair $(\sigma,d)$, and moreover $\Sp(d,\sigma)$ appears with multiplicity one in $L_d^{2,\infty}(\A_E^\times \GL_n(E)\backslash \GL_n(\A_E),\omega)$  (this of course was already known for $d=1$ by the pioneering independent results of Piatetski-Shapiro and Shalika). 

We define the spaces $L^{2,\infty}(\SL_n(E)\backslash \SL_n(\A_E))$ and $\mathcal{A}^{\infty}_0(\SL_n(E)\backslash \SL_n(\A_E))$ in the same way we defined their $\GL$-analogues. Also, similarly the notions of square-integrable and cuspidal automorphic representations of $\SL_n(\A_E)$ are defined.

We recall from \cite[Chapter 4]{hs12} (see in particular \cite[Remark 4.23]{hs12} for square integrable representations) the following facts. If $\pi$ is a square-integrable automorphic representation of $\SL_n(\A_E)$, then there is a square-integrable automorphic representation $\widetilde{\pi}$ of $\GL_n(\A_E)$ such that $\pi$ is a submodule of $\Res(\widetilde{\pi})$ where $\Res$ is the restriction of automorphic forms from 
$\GL_n(\A_E)$ to $\SL_n(\A_E)$, and such a $\widetilde{\pi}$ is unique up to twisting by an automorphic character of $\A_E^\times$. We denote by $\La(\pi)$ or $\La(\widetilde{\pi})$ the set of irreducible submodules of 
$\Res(\widetilde{\pi})$.

\subsection{Degenerate Whittaker models and square-integrable $\La$-packets}\label{section global degenerate whittaker models}

Let $n=dr$. Let $\sigma$ be a smooth unitary cuspidal automorphic representation of $\GL_r(\A_E)$ and let 
$\widetilde{\pi}=\Sp(d,\sigma)$ be the associated square-integrable automorphic representation of $\GL_n(\A_E)$.
We set $U_{r^d}$ to be the unipotent radical of the parabolic subgroup of type $r^d$ of $\GL(n)$, denoted by $P_{r^d}$. Let 
\[\psi_{1,\dots,d}(\diag(n_1,\dots,n_d)u)=\prod_{i=1}^{d} \psi_i(n_i)\] where $\psi_i$ is a non-degenerate character of 
$N_r(\A_E)$ trivial on $N_r(E)$ and $u\in U_{r^d}(\A_E)$. For $\varphi\in \pi$, we set 
\[p_{\psi_{1,\dots,d}}(\varphi)=\int_{N_n(E)\backslash N_n(\A_E)} \varphi(n) \psi_{1,\dots,d}^{-1}(n) dn .\]

By \cite[Corollary 3.4]{jl13}, there exists $\varphi\in \Sp(d,\sigma)$ such that 
$p_{\psi_{1,\dots,d}}(\varphi)\neq 0$: we will say that $\varphi$ \textit{has a non-zero Fourier coefficient of type} $r^d$ or 
\textit{a degenerate Whittaker model of type} $r^d$. Of course when $d=1$ this result is due to the pioneering works of Piatetski-Shapiro and Shalika. 

\begin{remark} The result \cite[Corollary 3.4]{jl13} could also be deduced by the techniques used in Section \ref{section global distinction}, using 
the $E=F\times F$-analogue of Yamana's formula \cite[Theorem 1.1]{yam15} (see Theorem \ref{theorem twisted Yamana}). Also following Section \ref{section global distinction} in the case where $E$ is split, one would conclude that 
any square-integrable representation of $\SL_n(A_E)$ in the $\La$-packet determined by $\Sp(d,\sigma)$ has a degenerate Whittaker model of type $r^d$. However for the sake of variety we offer a different proof of this fact here, using the results 
of \cite{jl13} rather than those of \cite{yam15} (or rather its split analogue).
\end{remark}

\textit{We say that a square-integrable representation $\pi$ of $\SL_n(\A_E)$ is of type $r^d$ if it belongs to $\La(\Sp(d,\sigma))$ for an irreducible (unitary) cuspidal automorphic representation $\sigma$ of $G_r(\A_E)$.}

We say that $\widetilde{\pi}$ (resp. $\pi$) has a degenerate Whittaker model of type $r^d$ if
 there is $\varphi\in \widetilde{\pi}$ (resp. $\varphi\in\pi$) with a non-zero Fourier coefficient of type $r^d$. In particular 
 $\Sp(d,\sigma)$ has a degenerate Whittaker model of type $r^d$.

We denote by $\psi$ a non-degenerate character of $N_r(\A_E)$ trivial on 
 $N_r(E)$. We set 
 \[(\mathbf{1}\otimes \psi)\begin{pmatrix} I_{n-r} & x \\ & u_1 \end{pmatrix}=\psi(u_1)\] for 
 \[\begin{pmatrix} I_{n-r} & x \\ & u_1 \end{pmatrix}\in U_{n,r}(\A_E).\] For $\varphi\in \widetilde{\pi}$, we set 
 \[\varphi_{U_{n,r},\psi}(g)=\int_{U_{n,r}(E)\backslash U_{n,r}(\A_E)}\varphi(ug)(\mathbf{1}\otimes \psi^{-1})(u)du\] for $g \in \GL_n(\A_E)$. 
 
\begin{remark}\label{remark constant term}
Note that the function $\varphi_{U_{n,r},\psi}$ is nothing but the integral of the constant term of $\varphi$ along the $(n-r,r)$ parabolic against $\psi^{-1}$ on $N_r(E) \backslash N_r(\A_E)$. By \cite[Lemma 6.1]{yam15} there is a positive character $\d$ of $\GL_{n-r}(\A_E)$ such that 
the function $\delta\otimes \varphi_{U_{n,r},\psi}$ belongs to $\Sp(d-1,\sigma)$, in particular  $(\varphi_{U_{n,r},\psi})_{|H}$ belongs to $\Res_H(\Sp(d-1,\sigma))$ (restriction of cusp forms) for any subgroup $H$ of $\GL_n(\A_E)^1$, for example 
$H=\SL_n(\A_E)$.
\end{remark}

We now can prove the following result.
 
\begin{proposition}\label{proposition degenerate whittaker model inside global L packets}
A square-integrable automorphic representation $\pi$ of $\SL_n(\A_E)$ of type $r^d$ has a degenerate Whittaker model of type $r^d$. 
\end{proposition}

\begin{proof}
We will prove the stronger claim: for any $\varphi\in \widetilde{\pi}$ such that $\varphi_{|\SL_n(\A_E)}\neq 0$, there is $h_0\in \SL_r(\A_E)$ (embedded in $\SL_n(\A_E)$ in the upper left block) such that 
$\rho(h_0)\varphi$ has a non-zero Fourier coefficient of type $r^d$. If $d=1$, we are in the cuspidal (and hence generic) case and the result follows from the 
same inductive procedure of Lemma \ref{lemma step 1} and Proposition \ref{proposition unfolding}, but applied to $E$ diagonally embedded inside $E\times E$ (instead of $F\subset E$ considered there). If $d\geq 2$, by \cite[Proposition 3.1 (1)]{jl13} applied to $\varphi$ there is a non-degenerate character $\psi$ of $N_r(\A_E)$ trivial on $N_r(E)$ such that $\varphi_{U_{n,r},\psi}$ is non-zero on $\SL_{n-r}(\A_E)$ (because $N_{n,r}(E)\backslash  P_n(E)=N_{n,r}^\circ(E)\backslash  P_n^\circ(E)$ as $d\geq 2$). We conclude by induction, thanks to Remark \ref{remark constant term}.
\end{proof}
 
We have the following corollary 
of Proposition \ref{proposition degenerate whittaker model inside global L packets}

\begin{corollary}\label{corollary transitive action on L packets}
If $\widetilde{\pi}$ is an irreducible square-integrable automorphic representation of $\GL_n(\A_E)$ of type $r^d$, then $\Res(\widetilde{\pi})$ is multiplicity free. Moreover, for any automorphic character $\psi$ of $N_n(\A_E)$ of type $r^d$, the $\La$-packet
$\La(\widetilde{\pi})$ contains a unique member $\pi(\psi)$ with a $\psi$-Whittaker model, and the group $\diag(E^\times,I_{n-1})$ acts transitively on $\La(\widetilde{\pi})$.
\end{corollary}
\begin{proof}
Thanks to multiplicity one inside local L-packets (cf. Proposition \ref{proposition multiplicity one inside non arch L-packets} and Corollary \ref{corollary existence of deg Whitt model and multiplicity one inside arch L packets}), it follows that the representations in 
$\La(\widetilde{\pi})$ appear with multiplicity one in $\Res(\widetilde{\pi})$. Moreover, we deduce that $T_n(E)$ acts transitively on $\La(\widetilde{\pi})$: by Proposition \ref{proposition degenerate whittaker model inside global L packets} any  representation in $\La(\widetilde{\pi})$ has a degenerate Whittaker model of type $r^d$. Note that two automorphic characters of type $r^d$ of $N_n(\A_E)$ are conjugate to each other by $T_n(E)$ and this implies that for each automorphic character $\psi$ of type $r^d$ of $\N_n(\A_E)$ there is a representation $\pi(\psi)$ in $\La(\widetilde{\pi})$ with a $\psi$-Whittaker model. Moreover $\La(\widetilde{\pi})$ has at most one representation with a $\psi$-Whittaker model by local multiplicity one of degenerate Whittaker models and this implies the uniqueness of $\pi(\psi)$ in the statement. Finally for $t\in T_n(E)$ and $t^\prime=\diag(\det(t),I_{n-1})$, the representations 
$\pi^t$ and $\pi^{t^\prime}$ in $\La(\widetilde{\pi})$ are isomorphic, hence equal by multiplicity one inside $\La(\widetilde{\pi})$.
\end{proof}

\subsection{Distinguished representations and distinguished L-packets}

For $\chi\in \widehat{F^\times\backslash \A_F^\times}$ and $\omega \in \widehat{E^\times\backslash \A_E^\times}$ extending $\chi$, we denote by $\widetilde{p}_{n,\chi}$ the linear form called $\chi$-period integral on $L_d^{2,\infty}(\A_E^\times \GL_n(E)\backslash \GL_n(\A_E),\omega)$ given by 
\[\widetilde{p}_{n,\chi}(\phi)= \int_{\A_F^\times\GL_n(F)\backslash \GL_n(\A_F)}\phi(h)\chi^{-1}(\det(h))dh.\] It is 
well-defined on $\mathcal{A}^{\infty}_0(\A_E^\times\GL_n(E)\backslash \GL_n(\A_E),\omega)$ by \cite[Proposition 1]{agr93} and 
in general by \cite[Lemma 3.1]{yam15}. Indeed up to a positive constant 
$\widetilde{p}_{n,\chi}(\phi)$ is equal to \[\int_{\GL_n(F)\backslash \GL_n(\A_F)^1}\phi(h)\chi^{-1}(\det(h))dh.\] 
\textit{We say that a square-integrable automorphic representation 
\[\widetilde{\pi}\subset L_d^{2,\infty}(\A_E^\times \GL_n(E)\backslash \GL_n(\A_E),\omega)\] is $\chi$-distinguished (or simply distinguished when $\chi\equiv 1$) if $\widetilde{p}_{n,\chi}$ is non-vanishing on $\widetilde{\pi}$}. 

We denote by $p_n$ the period integral on $L_d^{2,\infty}(\SL_n(E)\backslash \SL_n(\A_E))$ given by 
\[p_n(\phi)= \int_{\SL_n(F)\backslash \SL_n(\A_F)}\phi(h)dh.\]
It is again well-defined on $\mathcal{A}^{\infty}_0(\SL_n(E)\backslash \SL_n(\A_E))$ thanks to \cite[Proposition 1]{agr93} 
and on the space $L_d^{2,\infty}(\SL_n(E)\backslash \SL_n(\A_E))$ by the arguments in \cite[Lemma 3.1]{yam15}, and 
\textit{we say that a square-integrable representation $\pi\subset L_d^{2,\infty}(\SL_n(E)\backslash \SL_n(\A_E))$ is 
distinguished if $p_n$ does not vanish on $\pi$}. 
We give another useful formula for the $\SL_n$-period integral following \cite[Proposition 3.2]{ap06}. 

\begin{proposition}\label{proposition fourier}
Let $\widetilde{\pi}$ be a square integrable automorphic representation of $\GL_n(\A_E)$. The period integral 
\[\varphi\mapsto \int_{\SL_n(F)\backslash \SL_n(\A_F)}\varphi(h)dh\] 
is given by an absolutely convergent integral on $\Res(\widetilde{\pi})$. Moreover, for any $\varphi\in \widetilde{\pi}$, we have
\[\int_{\SL_n(F)\backslash \SL_n(\A_F)}\varphi(h)dh=\sum_{\alpha} \int_{\GL_n(F)\backslash \GL_n(\A_F)^1}\varphi(h)\alpha(\det(h))dh\] where the sum is over all characters $\alpha$ of the compact abelian group $F^\times \backslash \A_F^1$.
\end{proposition}
\begin{proof}
For the absolute convergence of the integrals, the arguments of \cite[Lemma 3.1]{yam15} adapt in a straightforward manner and we do not repeat them. 
The proof of the second point is now essentially that of \cite[Proposition 3.2]{ap06}. Indeed 
\[\int_{\GL_n(F) \backslash \GL_n(\A_F)^1} \varphi(h)dh = \int_{F^\times \backslash \A_F^1} \left( \int_{\SL_n(F)\backslash \SL_n(\A_F)} \varphi(h \diag(x,I_{n-1}))dh\right) dx\] and one applies Fourier inversion on the compact abelian group $F^\times \backslash \A_F^1$.
\end{proof}

\begin{remark}\label{remark finiteness}
We may remark here that the sum of the $(\GL(n,\A_F)^1,\alpha)$-periods over all characters $\alpha$ of the group $F^\times\backslash \A_F^1$ is in fact a finite sum. To this end, we may assume $\widetilde{\pi}$ is distinguished with respect to $\GL_n(\A_F)$ (indeed if $\widetilde{\pi}$ is $(\GL_n(\A_F)^1,\alpha)$-distinguished, then it is 
$(\GL_n(\A_F),\alpha')$-distinguished for a character $\alpha'$ of $\A_F^\times$ extending $\alpha$ and 
equal to the central character $\omega_{\widetilde{\pi}}$ on $(\A_F)_{>0}$). Observe that this implies that $\widetilde{\pi}$ is Galois conjugate self-dual by strong multiplicity one for the residual spectrum \cite{mw89} and the fact that $\Sp(d,\sigma_v)$ is distinguished and hence Galois conjugate self-dual for any finite place $v$ (\cite{fli91}). Now, if the $(\GL(n,\A_F)^1,\alpha)$-period is also non-zero then we have $\widetilde{\pi} \cong \widetilde{\pi} \otimes \alpha' \circ N_{E/F}$ for some $\alpha'$ extending $\alpha$ to $\A_F^\times$. As $\widetilde{\pi} = \Sp(d,\sigma)$, we see that $\sigma \cong \sigma \otimes \alpha' \circ N_{E/F}$. As $\sigma$ is a cuspidal representation, the finiteness of the set of such characters $\alpha'$ (hence of that of the characters $\alpha$) follows from \cite[Lemma 3.6.2]{ram00} (which is \cite[Lemma 4.11]{hs12}).    
\end{remark}

\begin{definition}
We say that the $\La$-packet determined by a square-integrable representation of $\GL_n(\A_E)$ is distinguished if it contains a distinguished representation of $\SL_n(\A_E)$.
\end{definition}

\section{Distinction inside global L-packets}\label{section global distinction}

The aim of this section is to establish our main result, namely Theorem \ref{theorem main} which asserts that distinguished representations inside distinguished L-packets are those with a degenerate $\psi$-Whittaker model for some 
distinguished $\psi$, and to give a first application of of it (Theorem \ref{theorem SL analogue of yam thm 1.2}). The proof is an induction based on the unfolding method, and has two steps, the first one being the cuspidal step (corresponding to $d=1$).

\subsection{The cuspidal case}\label{subsection period vs Whittaker period}

Here we characterize members of distinguished $\La$-packets of $\SL_n(\A_E)$ with non-vanishing $\SL_n(A_F)$-period in terms of Whittaker periods. The following lemma is a generalization of 
\cite[Lemma 4.3]{ap06}, but the proof in \cite[Lemma 4.3]{ap06} does not generalize to this case. We denote by $Q_n$ the proper parabolic subgroup of $\SL_n$ containing $P_n^\circ=\SL_{n-1}.U_n$. For $n\geq 3$, we set \[R_n=\{\diag(x,I_{n-2},x^{-1}),x\in \mathbb{G}_m\},\] so $Q_n$ is the semi-direct product $P^1_n.R_n$.

\begin{lemma}\label{lemma step 1}
Take $n\geq 3$. Let $\varphi$ be a cusp form on $\SL_n(\A_E)$ such that 
\[\int_{\SL_n(F)\backslash \SL_n(\A_F)} \varphi(h) dh\neq 0,\] then there is $h_0\in \SL_n(\A_F)$ (and in fact in $R_n(\A_F)$) such that \[\int_{P_n^1(F)\backslash P_n^1(\A_F)} \varphi(hh_0) dh\neq 0,\] where 
this integral is absolutely convergent. 
\end{lemma}
\begin{proof}
According to \cite[Section 18.2]{sv17}, there is $s\in \mathbb{C}$ such that for $\Re(s)$ large enough, the integral $\int_{Q_n(F)\backslash Q_n(\A_F)} \varphi(p)\d_{Q_n}^s(p) dp$ is absolutely convergent. Moreover it has meromorphic continuation, 
and there is a meromorphic function $r(s)$ with $r(0)=0$ such that $r(s)\int_{Q_n(F)\backslash Q_n(\A_F)} \varphi(h)\d_{Q_n}^s(h) dh$ tends to $\int_{\SL_n(F)\backslash \SL_n(\A_F)} \varphi(h) dh\neq 0$ when $s\rightarrow 0$. In particular there is an $s\in \R$ large enough in the realm of absolute convergence such that
\[0\neq \int_{Q_n(F)\backslash Q_n(\A_F)} \varphi(p)\d_{Q_n}^s(p) dp = \int_{P_n^1(F)\backslash P_n^1(\A_F)}\int_{R_n(F)\backslash R_n(\A_F)} \varphi(pa)\d_{Q_n}^s(a) dpda\] hence there is an $a\in R_n(\A_F)$ such that  
$\d_{Q_n}^s(a) \int_{P_n^1(F)\backslash P_n^1(\A_F)}\varphi(pa)dp \neq 0$ and the result follows.
\end{proof}

\begin{remark}
A result similar to Lemma \ref{lemma step 1} is \cite[Proposition 8]{dp19} where it is proved via unfolding an Eisenstein series $E(h,s)$ on $\SL_n(\A_F)$ and using that 
\[{\rm Res}_{s=1} \left(\int_{\SL_n(F)\backslash \SL_n(\A_F)} \varphi(h)E(h,s)dh \right) = \mathcal{P}_{\SL_n(\A_F)}(\varphi),\] a trick that \cite{dp19} attributes to \cite{agr93}. A straightforward adaptation of the proof of \cite[Proposition 8]{dp19} can also be used to prove Lemma \ref{lemma step 1}. Though our proof here looks much shorter where we appeal to \cite[Section 18.2]{sv17}, however the core of \cite[Proposition 18.2.1]{sv17} is the equality (18.6) and what follows in {\it loc. cit.}, and it relies on the exact same considerations on Eisenstein series as in \cite[Proposition 8]{dp19}. Hence the proof above is in fact essentially the same as that of 
\cite[Proposition 8]{dp19} but the main part of the argument is contained in the statement of \cite[Section 18.2]{sv17}. Note that \cite[Section 18.2]{sv17} is done in general for any semisimple group.
\end{remark}

We recall that $U_{n,k}=U_n \dots U_{k+1}<N_n=U_{1,n}.$ For $\psi_{n,k}$ a character of 
$U_{k,n}(\A_E)$ and $\varphi$ a cusp form on 
$\SL_n(\A_E)$, we set 
\[\varphi_{\psi_{n,k}}(x)=\int_{N_{n,k}(E)\backslash N_{n,k}(\A_E)} \varphi(nx)\psi_{n,k}^{-1}(n)dn\] for $x\in \SL_n(\A_E)$. When $k=1$ and $\psi_{n,1}$ is non-degenerate, we write 
$\varphi_{\psi_{n,1}}=W_{\varphi,\psi_{n,1}}$. The reader familiar with it will recognize what is often called the unfolding method in the following proof (see \cite[Section 6]{js90} for a famous and difficult instance of this technique).

\begin{proposition}\label{proposition unfolding}
Let $\varphi$ be a cusp form on $\SL_n(\A_E)$ such that 
\[\int_{P_n^1(F)\backslash P_n^1(\A_F)} \varphi(h) dh\neq 0,\] then there is a non-degenerate character $\psi$ of $N_n(\A_E)/N_n(E+\A_F)$ such that 
$W_{\varphi,\psi}$ does not vanish on $\SL_n(\A_F)$. In particular thanks to Lemma \ref{lemma step 1} if $\pi$ is an $\SL_n(\A_F)$-distinguished cuspidal automorphic representation of $\SL_n(\A_E)$, 
then it is $\psi$-generic for a non-degenerate character $\psi$ of $N_n(\A_E)/N_n(E+\A_F)$.
\end{proposition}
\begin{proof}
We do an induction on $n$, the case $n=2$ being part of the proof of \cite[Theorem 4.2]{ap06}. Hence we suppose $n\geq 3$. 
By hypothesis we have \[\int_{\SL_{n-1}(F)\backslash \SL_{n-1}(\A_F)} \int_{U_n(F)\backslash U_n(\A_F)}\varphi(uh) du dh\neq 0.\] 
Set \[\varphi^{U_n,F}(x)=\int_{U_n(F)\backslash U_n(\A_F)}\varphi(ux)du \] for $x\in  \SL_{n-1}(\A_F)$. 
By Poisson formula for $(F\backslash \A_F)^{n-1}\subset (E\backslash \A_E)^{n-1}$, we have 
\[\varphi^{U_n,F}(x)=\sum_{\psi_{n,n-1} \in  \widehat{\frac{U_n(\A_E)}{U_n(E+\A_F)}}} \varphi_{\psi_{n,n-1}}(x),\] which is in turn equal to \[\sum_{\psi_{n,n-1} \in  \widehat{\frac{U_n(\A_E)}{U_n(E+\A_F)}}-\{1\}} \varphi_{\psi_{n,n-1}}(x)\] by 
cuspidality of $\varphi$. The convergence of the series is absolute (and can be shown to be uniform for $x$ in compact subsets of 
$\SL_{n-1}(\A_F)$ but we will not use it). For fixed non-degenerate $\psi_{n,n-1}^0$ of $U_n(\A_E)/U_n(E+\A_F)$, one has \[\varphi^{U_n,F}(x)=\sum_{\psi_{n,n-1} \in  \widehat{\frac{U_n(\A_E)}{U_n(E+\A_F)}}} \varphi_{\psi_{n,n-1}}(x)= 
\sum_{\gamma\in  P_{n-1}^\circ(F)\backslash \SL_{n-1}(F)} \varphi_{\psi_{n,n-1}^0}(\gamma x)\] 
because as $n\geq 3$, the group 
$\SL_{n-1}(F)$ acts transitively on the set of non-trivial characters of $U_n(\A_E)$ trivial on $U_n(E+\A_F)$, and the stabilizer of 
$\psi_{n,n-1}^0$ is $ P_{n-1}^\circ(F)$. Hence 
\[0\neq \int_{\SL_{n-1}(F)\backslash \SL_{n-1}(\A_F)} \int_{U_n(F)\backslash U_n(\A_F)}\varphi(uh) du dh=\int_{P_{n-1}^\circ(F)\backslash \SL_{n-1}(\A_F)} \varphi_{\psi_{n,n-1}^0}(h)dh\] where the right hand side is absolutely convergent (by Fubini). Now 
\[\int_{P_{n-1}^\circ(F)\backslash \SL_{n-1}(\A_F)} \varphi_{\psi_{n,n-1}^0}(h)dh= \int_{P_{n-1}^\circ(\A_F)\backslash \SL_{n-1}(\A_F)} \int_{P_{n-1}^\circ(F)\backslash P_{n-1}^\circ(\A_F)} \varphi_{\psi_{n,n-1}^0}(hx)dhdx,\] and 
this implies that \[\int_{P_{n-1}(F)\backslash P_{n-1}^\circ(\A_F)} \varphi_{\psi_{n,n-1}^0}(hh_0)dh\neq 0\] for some $h_0\in 
\SL_{n-1}(\A_F)$. The function $\varphi_0=(\rho(h_0)\varphi)_{\psi_{n,n-1}^0}=\rho(h_0)\varphi_{\psi_{n,n-1}^0}$ is a cusp form on $\SL_{n-1}(\A_E)$, and we can apply our induction hypothesis to it, to conclude that $W_{\varphi_0,\psi'}$ is non-zero on $\SL_{n-1}(\A_F)$ for some 
non-degenerate character $\psi'$ of $N_{n-1}(\A_E)$ trivial on $N_{n-1}(\A_F+E)$. Setting 
$\psi:=\psi'\otimes \psi_{n,n-1}^0:n'.u\mapsto \psi'(n')\psi_{n,n-1}^0(u)$, one checks that by definition: 
\[W_{\varphi_0,\psi'}(x)=W_{\rho(h_0)\varphi,\psi}(x)=W_{\varphi,\psi}(xh_0)\] for $x\in \SL_{n-1}(\A_E)$. The result follows.
\end{proof}

\begin{remark}\label{dp-tams}
As mentioned in \S \ref{intro} our strategy in proving Proposition \ref{proposition unfolding} is to have an inductive set up to reduce the proof to the case of $n=2$. In the finite field cuspidal case as well as in the $p$-adic field tempered case such an inductive machinery can be set up via Clifford theory \cite[Proposition 1]{dp19} and this is carried out in \cite[Proposition 4.2 \& Remark 4]{ap18}. A similar approach in the number field case can be carried out as well by making use of the global analogue of \cite[Proposition 1]{dp19} which is \cite[Proposition 6]{dp19}. This was brought to our attention by Dipendra Prasad. In fact \cite[Proposition 6]{dp19} is stated more generally and our inductive set up would follow by taking $H = SL_{n-1}(\A_F)$ and $A =  \frac{U_n(\A_E)}{U_n(E+\A_F)}$, in the notations of \cite[Proposition 6]{dp19}.
\end{remark}

\begin{remark}
Though not relevant to this paper, we remark here that the inductive strategy in the finite cuspidal and $p$-adic tempered cases mentioned in Remark \ref{dp-tams} do not seem to generalise to cover all the generic representations. However, the final result, that distinction is characterised by genericity for a non-degenerate character of $N(E)/N(F)$, is established via other methods. In the $p$-adic case, this is done in \cite{ap18}, and this we have further generalised in Theorem \ref{theorem0} of the present paper. In the finite field case, the general result is established in \cite{am18}.  
\end{remark}

\begin{remark} We cease the occasion to fill a small gap in the literature, which uses the ideas of this paper: namely the unfolding of the Asai $L$-function. The proofs given in \cite[p. 303]{fli88} and \cite[p. 558]{z14} are a bit quick. Here we add the details to the proof of [Flicker, 2 Proposition, p. 303]. The transition between the second and third line of the equality there relies on the following step: take $\varphi$ a cusp form on $\GL_n(\A_E)$, then 
\[\int_{N_n(F)\backslash N_n(\A_F)}\varphi(n) dn = \sum_{\gamma\in N_n(F)\backslash P_n(F)}W_{\varphi,\psi}(\gamma),\] where both the "integrals" are absolutely convergent and $\psi$ is a non-degenerate character of $N_n(\A_E)$ trivial on $N_n(\A_F+E)$. We use the same notations as in Proposition \ref{proposition unfolding}, and denote by $\psi_{n,n-1}^0$ the restriction of $\psi$ to $U_n(\A_E)$. 

Let us write 
\[\int_{N_n(F)\backslash N_n(\A_F)}\varphi(n) dn=\int_{N_{n-1}(F)\backslash N_{n-1}(\A_F)}\varphi^{U_n,F}(n) dn.\] 
By induction applied to the cusp form $\varphi^{U_n,F}$ on $\GL_{n-1}(\A_E)$, we have 
\[\int_{N_{n-1}(F)\backslash N_{n-1}(\A_F)}\varphi^{U_n,F}(n) dn = \sum_{\gamma'\in N_{n-1}(F)\backslash P_{n-1}(F)}\varphi^{U_n,F}(\gamma').\]
Now replace 
$\varphi^{U_n,F}(\gamma')$ by $\sum_{\gamma\in  P_{n-1}(F)U_n(F)\backslash P_n(F)} \varphi_{\psi_{n,n-1}^0}(\gamma \gamma')$ this time (still by Poisson formula and because $P_n(F)$ also acts transitively on the set of non-trivial characters of $U_n(\A_E)$ trivial on $U_n(E+\A_F)$, the stabilizer of $\psi_{n,n-1}^0$ being $ P_{n-1}(F)U_n(F)$). We get 
\begin{align*}
\int_{N_n(F)\backslash N_n(\A_F)}\varphi(n) dn &= \sum_{\gamma'\in  N_{n-1}(F)\backslash P_{n-1}(F)}  \sum_{\gamma \in  P_{n-1}(F)U_n(F)\backslash P_n(F)}W_{\varphi_{\psi_{n,n-1}^0},\psi_{|N_{n-1}(\A_E)}}(\gamma \gamma')dn \\
&= \sum_{\gamma'\in  N_{n-1}(F)\backslash P_{n-1}(F)}  \sum_{\gamma \in  P_{n-1}(F)U_n(F)\backslash P_n(F)} W_{\varphi,\psi}(\gamma\gamma') \\
&= \sum_{\gamma\in  N_{n-1}(F)U_n(F)\backslash P_n(F)}W_{\varphi,\psi}(\gamma),
\end{align*}
which is what we wanted.
\end{remark}

\subsection{The square-integrable case}\label{subsection square integrable distinction insideL packets}

Our aim in this section is to show that if $\pi$ is distinguished then $\pi$ has a non-vanishing Fourier coefficient with respect to a character of type $r^d$ of $N_n(\A_E)$ which is trivial on $N_n(E+\A_F)$ (see Proposition \ref{prop main}). The key ingredient in achieving this is Proposition \ref{proposition yamana-sln} below. 
 
The following result is \cite[Theorem 1.1]{yam15} slightly reformulated for our purposes.
 
\begin{theorem}\label{theorem twisted Yamana}
Let $n=rd$ with $r \geq 2$ and $d\geq 2$, and let $\psi$ be a non-degenerate unitary character of $N_n(\A_E)$ trivial on $N_n(E+\A_F)$. Fix a character $\alpha$ of $F^\times\backslash \A_F^1$. Then for $\varphi \in \widetilde{\pi} =  \Sp(d,\sigma)$, we have 
\[\int_{\GL_n(F)\backslash \GL_n(\A_F)^1} \varphi(h)\alpha(\det h)dh= \]
\[\int_{N_{n-1,r-1}^\circ(\A_F)\backslash \SL_{n-1}(\A_F)}\int_{\GL_{n-r}(F)\backslash\GL_{n-r}(\A_F)^1}
(\alpha \varphi)_{U_{n,r},\psi}(\diag(m,I_r)\diag(h,1)) dm dh.\] 
\end{theorem}
\begin{proof}
We denote by $\omega_{\sigma}$ the central character of $\sigma$. We extend $\alpha$ as $\alpha_0$ to $\A_F^\times$.  We then extend $\alpha_0$ to an automorphic character of $\beta$ of $\A_E^\times$. Then we claim that the following equality holds 
\[\int_{\GL_n(F)\backslash \GL_n(\A_F)^1} \varphi(h)\alpha_0(\det h)dh= \]
\[\int_{N_{n-1,r-1}(\A_F)\backslash \GL_{n-1}(\A_F)}\int_{\GL_{n-r}(F)\backslash\GL_{n-r}(\A_F)^1}
(\alpha_0 \varphi)_{U_{n,r},\psi}(\diag(m,I_r)\diag(h,1)) dm dh.\] 
Indeed, if $\alpha_0^r\cdot {\omega_{\sigma}}_{|\A_F^\times}$ is trivial, then this follows from the second part of Theorem \cite[Theorem 1.1]{yam15} applied to $\beta\otimes \pi$. If $\alpha_0^r\cdot {\omega_{\sigma}}_{|\A_F^\times}\not \equiv 1$, then it follows from the first part of \cite[Theorem 1.1]{yam15} applied to $\beta\otimes \pi$, with the extra observation that the right hand side of the equality also vanishes, thanks to Remark \ref{remark constant term} and the first part of Theorem \cite[Theorem 1.1]{yam15} again if $d\geq 3$, and for central character reasons when $d=2$. We can now replace the quotient $N_{n-1,r-1}(\A_F)\backslash \GL_{n-1}(\A_F)$
by $N_{n-1,r-1}^\circ(\A_F)\backslash \SL_{n-1}(\A_F)$ and the statement follows.
\end{proof}
 
From Theorem \ref{theorem twisted Yamana}, we deduce its $\SL(n)$ version by making use of Proposition \ref{proposition fourier}.

\begin{proposition}\label{proposition yamana-sln}
With notations and assumptions ($r,d\geq 2$) as in Theorem \ref{theorem twisted Yamana}, for $\varphi \in \Res(\widetilde{\pi})$ we have 
\[p_n(\varphi)=\int_{N_{n-1,r-1}^\circ(\A_F)\backslash \SL_{n-1}(\A_F)}\int_{\SL_{n-r}(F)\backslash\SL_{n-r}(\A_F)}
 \varphi_{U_{n,r},\psi}(\diag(m,I_r)\diag(h,1)) dm dh.\]
\end{proposition}

\begin{proof}
We relate the $\SL(n,\A_F)$-period $p_n$ to the $(\GL(n,\A_F)^1,\alpha)$-periods via Proposition \ref{proposition fourier}.  Applying Theorem \ref{theorem twisted Yamana} to each summand of the sum over characters $\alpha$ of $F^\times\backslash \A_F^\times$ just selected, we once again apply Proposition \ref{proposition fourier} to the right hand side sum to conclude the proof.
\end{proof}

Setting  
\[(\rho(g)\varphi)_{n-r,\psi}:=m\in \GL_{n-r}(\A_E) \mapsto \varphi_{U_{n,r},\psi}(\diag(m,I_r)g),\] Proposition \ref{proposition yamana-sln} implies the following observation which we state as a lemma.  

\begin{lemma}\label{lemma induction}
With notations and assumptions ($r,d\geq 2$) as in Theorem \ref{theorem twisted Yamana}, suppose that $\varphi \in \Res(\widetilde{\pi})$ is such that $p_n(\varphi)\neq 0$, then 
there is $h\in \SL_{n-1}(\A_F)$ such that \[p_{n-r}((\rho(\diag(h,1))\varphi)_{n-r,\psi})\neq 0.\]
\end{lemma}

We now state the main theorem of this section.

\begin{theorem}\label{theorem main}
Let $\La(\widetilde{\pi})$ be a distinguished square integrable $\La$-packet of $\SL_n(\A_E)$ of type $r^d$. Then the period integral $p_n$ does not vanish on $\pi\in \La(\widetilde{\pi})$ if and only if there exists a degenerate character $\psi$ of type $r^d$ of $N_n(\A_E)$ trivial on $N_n(E+\A_F)$ such that $p_{\psi_{1,\dots,d}}$ does not vanish on $\pi$.
\end{theorem}

The key direction of Theorem \ref{theorem main} is Proposition \ref{prop main} which follows from Lemma \ref{lemma induction} by an inductive argument (see also the proof of Proposition \ref{proposition unfolding}). 

\begin{proposition}\label{prop main}
Let $\pi$ be an irreducible automorphic representation of $\SL_n(\A_E)$ of type $r^d$ which is distinguished 
with respect to $\SL_n(\A_F)$; thus there exists $\varphi\in \pi$ such that $p_n(\varphi)\neq 0$. 
Then there exist $d$ non-degenerate characters $\psi_i$ of $N_r(A_E)$ trivial on $N_r(E+\A_F)$ and $\varphi^\prime \in \pi$ such that 
\[p_{\psi_{1,\dots,d}}(\varphi') = \int_{N_n(E)\backslash N_n(\A_E)} \varphi'(n) \psi_{1,\dots,d}^{-1}(n) dn \neq 0.\]
Moreover, $\varphi^\prime$ can be chosen to be a right $\SL_{n-1}(\A_F)$-translate of $\varphi$. 
\end{proposition}

\begin{proof}
The theorem is immediate from Lemma \ref{lemma induction} by an inductive argument, however we have to treat the case $r=1$ separately. If $r=1$ then $\pi$ is the trivial character of $\SL_n(\A_E)$ and the claim is obvious. So we suppose that $r\geq 2$. If $d=1$ the result is proved in Proposition \ref{proposition unfolding}, so we assume $d\geq 2$. Since $\varphi \in \pi$ is such that 
$p_n(\varphi) \neq 0$, by Lemma \ref{lemma induction}, we get $h \in \SL_{n-1}(\A_F)$ such that $p_{n-r}( (\rho(h)\varphi)_{n-r,\psi}) \neq 0$. 
Therefore, by induction thanks to Remark \ref{remark constant term}, we get $d-1$ non-degenerate characters $\psi_i,\ i=2,\dots,d$ 
of $N_r(A_E)$ trivial on $N_r(E+\A_F)$ such that
\[p_{\psi_{2,\dots,d}}[\rho(x)(\rho(h)\varphi)_{n-r,\psi}]= \int_{N_{n-r}(E)\backslash N_{n-r}(\A_E)}
(\rho(h)\varphi)_{n-r,\psi}(nx) \psi_{2,\dots,d}^{-1}(n) dn \neq 0,\]
for some $x=$ diag$(y,1)$ for $y \in \SL_{n-r-1}(\A_F)$.
But setting $\psi_1:=\psi$,
\begin{flushleft}
$\int_{N_{n-r}(E)\backslash N_{n-r}(\A_E)} (\rho(h)\varphi)_{n-r,\psi_1}(nx) \psi_{2,\dots,d}^{-1}(n) dn $
\end{flushleft}
\begin{align*}
&= \int_{N_{n-r}(E)\backslash N_{n-r}(\A_E)}   \varphi_{U_{n,r},\psi}(\diag(nx,I_r)h)\psi_{1,\dots,d-1}^{-1}(n) dn \\
&= \int_{N_{n-r}(E)\backslash N_{n-r}(\A_E)} \int_{U_{n,r}(E)\backslash U_{n,r}(\A_E)}\varphi(u ~\diag(nx,I_r)h) \psi_{1,\dots,d-1}^{-1}(n)  (\mathbf{1}\otimes \psi^{-1})(u) dn du \\
&= \int_{N_n(E)\backslash N_n(\A_E)} \varphi(n ~\diag(x,I_r)h) )\psi_{1,\dots,d}^{-1}(n)dn,
\end{align*}
and the result follows.
\end{proof}

To end the proof of Theorem \ref{theorem main}, it now suffices to prove the following implication, which is part of the proof of \cite[Theorem 4.2]{ap06}, and which we repeat.

\begin{lemma}\label{lemmathm}
Let $\La(\widetilde{\pi})$ be a distinguished $\La$-packet of $\SL_n(\A_E)$ of type $r^d$. If $\pi \in \La(\widetilde{\pi})$ is $\psi$-generic with respect to a degenerate character $\psi$ of type $r^d$ of $N_n(\A_E)$ trivial on $N_n(E+\A_F)$, then $p_n$ does not vanish on $\pi$.
\end{lemma}
\begin{proof}
By definition there is $\pi' \in \La(\widetilde{\pi})$ such that $p_n$ does not vanish on it. By Proposition \ref{prop main}, the representation $\pi'$ is $\psi'$-generic for a degenerate character $\psi'$ of type $r^d$ of $N_n(\A_E)$ trivial on $N_n(E+\A_F)$. Now there is $t\in T_n(F)$ such that 
$\psi={\psi'}^t$ where ${\psi'}^t(n) = \psi'(t^{-1}n t)$. And then the representation ${\pi'}^t$ given by ${\pi'}^t(g) =\pi'(t^{-1}gt)$ appears in $\La(\widetilde{\pi})$ and 
is $\psi$-generic. We deduce that $\pi={\pi'}^t$, by the local uniqueness of 
degenerate Whittaker models, and the result follows since $t\in \GL_n(F)$.
\end{proof}

Let us now state a simple but very useful consequence of Theorem \ref{theorem main} whose proof idea we have already employed in the proof of Lemma \ref{lemmathm}. We formulate this with an application in Section \ref{section-higher} in mind.

\begin{corollary}\label{corollary crux}
Let $\pi$ be a square integrable automorphic $\SL_n(\A_F)$-distinguished representation of $\SL_n(\A_E)$ and let $\La(\widetilde{\pi}')$ be a distinguished $\La$-packet of $\SL_n(\A_E)$ containing an isomorphic copy of $\pi$. Then the period $p_n$ does not vanish on the unique representation in $\La(\widetilde{\pi}')$ isomorphic to $\pi$. 
\end{corollary}
\begin{proof}
Call $\pi'$ the isomorphic copy of $\pi$ in $\La(\widetilde{\pi}')$. Thanks to Theorem \ref{theorem main}, $\pi$ is $\psi$-generic for $\psi$ a distinguished degenerate character of $N_n(\A_E)$ trivial on $N_n(E+\A_F)$ of the correct type and therefore $\pi'$ has a locally $\psi_v$-degenerate Whittaker model for every place $v$ of $F$. By Theorem \ref{theorem main} again, 
the $\psi$-generic representation $\pi''$ in $\La(\widetilde{\pi}')$ is also 
$\SL_n(\A_F)$-distinguished. But thanks to multiplicity one of local degenerate Whittaker models, two locally $\psi$-generic automorphic representations in the same $\La$-packet are equal, hence 
$\pi'=\pi''$, and we deduce that $p_n$ does not vanish on $\pi'$. 
\end{proof}

As a corollary to Theorem \ref{theorem main}, we state and prove one more variation of the above theme. This is applied in Section \ref{section local global}.

\begin{proposition}\label{corollary transitive on dist}
Let $\pi$ be an irreducible square-integrable automorphic representation of $\SL_n(\A_E)$. The group $\diag(F^\times,I_{n-1})$ acts transitively on the set of distinguished members of $\La(\pi)$.
\end{proposition}
\begin{proof}
From Theorem \ref{theorem main} and the local uniqueness of 
degenerate Whittaker models, we easily deduce that $T_n(F)$ acts transitively on the set of distinguished members of $\La(\pi)$, and that the representations in 
$\La(\pi)$ appear with multiplicity one. However for $t\in T_n(F)$ and $t^\prime=\diag(\det(t),I_{n-1})$, the representations 
$\pi^t$ and $\pi^{t^\prime}$ in $\La(\pi)$ are isomorphic, hence equal by multiplicity one inside $\La(\pi)$.
\end{proof}

\subsection{Automorphy and distinction of the highest derivative for $\SL_n(\A_E)$}\label{section distinction of the highest derivative}

As a first application of Theorem \ref{theorem main}, we end this section with an analogue of \cite[Theorem 1.2]{yam15} 
in the context of $\SL_n(\A_E)$. 

\begin{lemma}\label{proposition automorphy of the global highest derivative}
Let $\pi$ be an irreducible square-integrable representation of $\SL_n(\A_E)$ of type $r^d$, and write \[\pi=\otimes'_v \pi_v.\] Then for any 
$k\in [1,d]$, the representation \[\pi^{[r^{d-k}]}(\psi_{d-k+1,\dots,d}):=\otimes'_v \pi_v^{[r^{d-k}]}(\psi_{{d-k+1,\dots,d},v})\] (see Definitions \ref{definition irrep with fixed wh model in the derivative} and \ref{definition irrep with fixed wh model in the archimedean derivative}) is automorphic. If $\sigma$ is a cuspidal automorphic representation of $\GL_r(\A_E)$ such that $\La(\pi)=\La(\Sp(d,\sigma))$, then $\pi^{[r^{d-k}]}(\psi_{d-k+1,\dots,d})$ is in fact the unique 
element of $\La(\Sp(k,\sigma))$ with a $\psi_{d-k+1,\dots,d}$-Whittaker model.
\end{lemma}
\begin{proof}
let $\mu$ be the member of $\La(\Sp(k,\sigma))$ with a $\psi_{d-k+1,\dots,d}$-Whittaker model. Then for all places $v$ 
the representation $\mu_v$ is the member of $\La(\Sp(k,\sigma_v))$ with a $\psi_v$-Whittaker model, hence it must be $\pi_v^{[r^{d-d}]}(\psi_{{d-k+1,\dots,d},v})$ and the result follows.
\end{proof}

Here is our $\SL$-analogue of \cite[Theorem 1.2]{yam15}.

\begin{theorem}\label{theorem SL analogue of yam thm 1.2}
Suppose that $\psi_{1,\dots,d}$ is a character of $N_n(\A_E)$ of type $r^d$ trivial on $N_n(E+\A_F)$. Let $\pi$ be an irreducible square-integrable representation of $\SL_n(\A_E)$ of type $r^d$ and fix $k\in [1,d]$, then $\pi(\psi_{1,\dots,d})$ is $\SL_n(\A_F)$-distinguished if and only if $\pi^{[r^{d-k}]}(\psi_{d-k+1,\dots,d})$ is 
$\SL_{kr}(\A_F)$-distinguished.
\end{theorem}
\begin{proof}
The proof is the same as that of Theorem \ref{theorem dist rep with given wh inside generic of Speh}, using \cite[Theorem 1.2]{yam15} in lieu of \cite[Theorem 2.13]{mat14}.
\end{proof}

\section{Characterization of distinguished square-integrable global L-packets}

Here we generalize the characterization of distinguished $\La$-packets given in \cite{ap06}, which turns out to be convenient in the proof of our main applications, namely the local-global principle inside distinguished L-packets of Section \ref{section local global} and the study of the behaviour of distinction with respect to higher multiplicity in Section \ref{questions}. The proof is based on the following well-known theorem which is a consequence of the works Jacquet and Shalika \cite{js81} on the one hand, and Flicker and Flicker-Zinoviev on the other hand \cite{fli88,fz95}. 

\begin{theorem}\label{theorem distinction vs conjugate duality}
Denote by $\omega_{E/F}$ the quadratic character attached to $E/F$ by global class field theory, and let $\widetilde{\pi}$ be a cuspidal automorphic representation of $\GL_n(\A_E)$. Then $\widetilde{\pi}$ is conjugate self-dual, i.e., $\widetilde{\pi}^\vee\simeq \widetilde{\pi}^\sigma$ if and only if $\pi$ is either distinguished or $\omega_{E/F}$-distinguished (and in fact not both together).
\end{theorem}
\begin{proof}
Let $\pi_1$, $\pi_2$ and $\pi_3$ be cuspidal automorphic representations of $\GL_n(\A_E)$. By the aforementioned references, the partial Rankin-Selberg $L^S(s,\pi_1,\pi_2)$ has a pole at $s=1$, which is necessarily simple, if and only if $\pi_2 \simeq \pi_1^\vee$, whereas the partial Asai $L$-function $L_{\mathrm{As}}^S(s,\pi_3)$ has a pole (necessarily simple) at $s=1$ if and only if $\pi_3$ is $\GL_n(\A_F)$-distinguished. The result now follows from the equality 
 \[L^S(s,\pi_1,\pi_1^\sigma)=L_{\mathrm{As}}^S(s,\pi_1)L_{\mathrm{As}}^S(s,\omega\otimes \pi_1)\]
 where $\omega$ is any Hecke character of $\A_E^\times$ extending $\omega_{E/F}$.
\end{proof}

First it implies the following lemma. 

\begin{lemma}\label{lemma H1 distinction vs H distinction}
Let $\alpha$ be a character of $F^\times \backslash \A_F^1$ and $\sigma$ be a cuspidal automorphic representation of $\GL_r(\A_E)$ with central character $\omega$. The restriction of $\omega$ to $(\A_F)_{>0}$ coincides with the restriction of 
$|\ |_{\A_F}^{ir\lambda}$ for a some $\l\in \R$, and we extend $\alpha$ to 
$\A_F^\times$ as the automorphic character $\alpha_{-\lambda}$. Suppose that the period integral
\[\widetilde{p}_{r,\alpha^{-1}}^1:\phi\mapsto \int_{\GL_r(F)\backslash \GL_r(\A_F)^1}\phi(h)\alpha(\det(h))dh\] is nonzero on $\sigma$, then $\alpha^r$ and $\omega^{-1}$ coincide $\A_F^1$ i.e. $(\alpha_{-\lambda}\circ \det)^{-1}$ restricts as $\omega$ to $\A_F^\times$, and $\sigma$ is $\alpha_{-\l}^{-1}$-distinguished, hence $\sigma^\vee\simeq \alpha_{-\l}\circ N_{E/F}\otimes \sigma^\theta$. 
\end{lemma}
\begin{proof}
The fact that $\alpha^r$ and $\omega^{-1}$ must coincide $\A_F^1$ if $\widetilde{p}_{r,\alpha^{-1}}^1$ does not vanish on 
$\widetilde{\pi}$ follows from central character considerations and the fact that $\widetilde{p}_{r,\alpha^{-1}}^1$ is
$\alpha^{-1}$-equivariant under $\GL_r(\A_F)^1$. But then for $\phi\in \widetilde{\pi}$ the function 
$\alpha_{-\l}\otimes \phi:g\mapsto \alpha_{-\l}(\det(g))\phi(g)$ is $\A_F^\times$-invariant and we conclude that 
$\widetilde{p}_{r,\alpha_{-\l}^{-1}}$ and $\widetilde{p}_{r,\alpha^{-1}}^1$ agree up to a positive constant, in particular 
$\sigma$ is $\alpha_{-\l}^{-1}$-distinguished. In particular 
for $\beta$ an automorphic character extending $\alpha_{-\l}$ to $\A_E^\times$, the representation $\beta\otimes \sigma$ is distinguished and we conclude that  $\sigma^\vee\simeq \alpha_{-\l}\circ N_{E/F}\otimes \sigma^\theta$ thanks to Theorem \ref{theorem distinction vs conjugate duality}.
\end{proof}

Now the characterization of square-integrable distinguished L-packets follows.

\begin{proposition}\label{proposition characterization of global distinguished L packets}
Let $\widetilde{\pi}=\Sp(d,\sigma)$ an irreducible square-integrable representation of $\GL_n(\A_E)$, with $\sigma$ a unitary cuspidal automorphic representation of $\GL_r(\A_E)$. Then $\La(\widetilde{\pi})$ is distinguished if and only if there is an automorphic character $\alpha\in \widehat{F^\times \backslash \A_F^\times}$ such that $\widetilde{\pi}^\vee\simeq \alpha\circ N_{E/F} \otimes \widetilde{\pi}^\theta$, or equivalently 
$\sigma^\vee \simeq \alpha\circ N_{E/F} \otimes \sigma^\theta$.
\end{proposition}
\begin{proof}
If $\widetilde{\pi}^\vee\simeq \alpha\circ N_{E/F} \otimes \widetilde{\pi}^\theta\Leftrightarrow\sigma^\vee\simeq \alpha\circ N_{E/F}\otimes \sigma^\theta$, then $\alpha\otimes \sigma$ is conjugate self-dual hence an automorphic twist of $\sigma$ distinguished by $\GL_r(\A_F)$ thanks to Theorem \ref{theorem distinction vs conjugate duality}. Hence by 
\cite[Theorem 1.2]{yam15} an automorphic twist of $\widetilde{\pi}$ is distinguished by $\GL_n(\A_F)$, and $\La(\widetilde{\pi})$ is distinguished thanks to Proposition \ref{proposition fourier} by a straightforward generalization of the second part of the proof of \cite[Proposition 3.2]{ap06}. Conversely if $\La(\widetilde{\pi})$ is distinguished, then by Proposition \ref{proposition fourier} and Lemma \ref{lemma H1 distinction vs H distinction}, an automorphic twist of $\widetilde{\pi}$ is distinguished and the result follows from Theorem \ref{theorem distinction vs conjugate duality}.
\end{proof}

\section{Local global principle for distinguished $L$-packets when $r$ is odd}\label{section local global}

This section establishes a local global principle for distinction inside a square-integrable $L$-packet of type $r^d$ of $\SL_n(\A_E)$, when $r$ is odd. 

Our proof makes use of the set up of \cite[Section 7]{ap13} where such a result is proved for a cuspidal $L$-packet of $\SL_2(\A_E)$. The proof over there is somewhat intricate and relied crucially on an analysis of the fibers of the Asai lift (see \cite[Remark in Section 7]{ap13}). Here our arguments are more elementary due to the fact that $r$ is odd. This is consistent with the earlier works \cite{ana05,ap18}. 

For the moment however $r$ is general. Let $\pi$ be an irreducible square-integrable automorphic representation of $\SL_n(\A_E)$ and denote by $\widetilde{\pi}$ a square-integrable automorphic representation of $\GL_n(\A_E)$ such that $\pi$ is realized in $\Res(\widetilde{\pi})$. 

We borrow the notations of \cite[Section 7]{ap13}. We consider $\A_E^\times$ as a subgroup of $\GL_n(\A_E)$ via $x \mapsto$ diag$(x,I_{n-1})$. This group acts by conjugation on isomorphism classes of an irreducible representation $\pi$ of $\SL_n(\A_E)$. The orbit of $\pi$ under this action is the representation theoretic $L$-packet of $\pi$, say $\La^\prime(\pi)$.  Let $G(\pi) < \A_E^\times$ be the stabilizer of $\pi$. Then, see \cite[p. 23]{hs12}, 
\[G_\pi = \bigcap_{\chi \in X(\widetilde{\pi})} {\rm Ker~} \chi \]
where 
\[X(\widetilde{\pi}) = \{\chi \in \widehat{E^\times \backslash \A_E^\times} \mid \widetilde{\pi} \otimes \chi \cong \widetilde{\pi}\}\]
which is a finite abelian group (cf. Remark \ref{remark finiteness}).

\begin{remark}\label{remark packets}
Note that $\La(\pi)$ consists of the automorphic members of $\La^\prime(\pi)$. Indeed clearly $\La(\pi)$ is included in this set. On the other hand, if $\pi'$ is an automorphic member of $\La^\prime(\pi)$, then it has a degenerate $\psi$-Whittaker model of type $r^d$ thanks to Proposition \ref{proposition degenerate whittaker model inside global L packets}. However $\La(\pi)$ also contains a member $\pi''$ with a degenerate $\psi$-Whittaker model according to Corollary \ref{corollary transitive action on L packets}. We conclude that $\pi'\simeq \pi''$ by local uniqueness of degenerate Whittaker models.
\end{remark}

We start with an elementary observation.
\begin{proposition}\label{proposition elementary}
Suppose $\widetilde{\pi}$ is a square-integrable automorphic representation of $\GL_n(\A_E)$ which is Galois conjugate self-dual; i.e., $\widetilde{\pi}^\vee \cong \widetilde{\pi}^\theta$, and that $\pi\in \La(\widetilde{\pi})$. Then $G_\pi$ is stable under the action of $\theta$.
\end{proposition}

\begin{proof}
As $\widetilde{\pi}$ is Galois conjugate self-dual, it follows that the finite abelian group $X(\widetilde{\pi})$ is stable under the Galois action, and thus $G_\pi$ is Galois stable. Alternatively, note that if $\pi_1$ and $\pi_2$ are in the same $L$-packet then $G_{\pi_1} = G_{\pi_2}$. Indeed, $\pi_2 = \pi_1^y$, for some $y \in \A_E^\times$, and by definition, $G_{\pi_2} = y^{-1}G_{\pi_1} y = G_{\pi_1}$ as the groups are abelian. In particular, $G_{\pi^\theta} = G_{\pi^\vee}$ as $\widetilde{\pi}^\vee \cong \widetilde{\pi}^\theta$. Observe also that $G_{\pi^\vee}=G_\pi$. Thus, if $x \in G_\pi$ then $x^\theta \in G_{\pi^\theta} = G_{\pi^\vee} = G_\pi$. 
\end{proof}

\emph{From now on, we assume that $E$ is split at the archimedean places, so that the archimedean analogue of 
Theorem \ref{theorem main p-adic} obviously holds}. 

As in \cite[Section 7]{ap13}, we define the following groups, 
\begin{eqnarray*}
H_0&=& \A_E^\times, \\
H_1 & = & \A_F^\times G_\pi, \\
H_2 & =& E^\times G_{\pi}, \\
H_3 & = &  F^\times G_{\pi},
\end{eqnarray*}
and we observe that  
\begin{enumerate}
\item The set $H_0 \cdot \pi$ is the $L$-packet of representations of $\SL_n(\A_E)$ determined by $\pi$ (see, for instance, \cite[Corollary 2.8]{hs12}).
\item The set $H_1 \cdot \pi$ is the set of locally distinguished representations in the $L$-packet of $\SL_n(\A_E)$ determined by $\pi$ (by Theorem \ref{theorem main p-adic} and its archimedean analogue). 
\item The set $H_2 \cdot \pi$ is the set of automorphic representations in the $L$-packet of $\SL_n(\A_E)$ determined by $\pi$ (by Proposition \ref{corollary transitive action on L packets}).
\item The set $H_3 \cdot \pi$ is the set of globally distinguished representations in the $L$-packet of $\SL_n(\A_E)$ determined by $\pi$ (by Proposition \ref{corollary transitive on dist}).
\end{enumerate}

We also record the following observation as a lemma.

\begin{lemma}\label{lemma x to r fixes pi}
Let $\pi$ as above be of type $r^d$, then for an $x\in \A_E^\times$, we have $x^r\in G_\pi$.
\end{lemma}
\begin{proof}
We observe that if $\pi$ has a $\psi_{1,\dots,d}$-Whittaker model with respect to 
the automorphic character $\psi_{1,\dots,d}$, then 
\[\pi^{\diag(xI_r,I_{n-r})}\in \La'(\pi).\] 
In particular, 
for finite places $v$, the local representation $\pi_v^{\diag(x_v I_r,I_{n-r})}$ has a $\psi_{{1,\dots,d},v}$-Whittaker model because $\diag(x_vI_r,I_{n-r})$ fixes $\psi_{{1,\dots,d},v}$ by conjugation, hence both 
$\pi_v$ and $\pi_v^{\diag(xI_r,I_{n-r})}$ have a $\psi_{{1,\dots,d},v}$-Whittaker model inside $\La(\pi_v)$, so they are equal, and the lemma follows.
\end{proof}

Next we state the local global principle for $(\SL_n(\A_E),\SL_n(\A_F))$ for square-integrable automorphic representations (for $r$ odd).

\begin{theorem}\label{theorem localglobal}
Let $\pi$ be an irreducible square-integrable automorphic representation of $\SL_n(\A_E)$ such that 
$\La(\pi)$ is distinguished. Assume that $r$ is odd and write $\pi =\otimes'_v \pi_v$ but this time for $v$ varying through the places of $F$ (hence here $\pi_v$ is $\pi_w$ for $w$ the place in $E$ lying over $v$ if $v$ does not split in $E$, and 
$\pi_v=\pi_{w_1}\otimes \pi_{w_2}$ if $v$ splits into $(w_1,w_2)$). Then, $\pi$ is distinguished with respect to $\SL_n(\A_F)$ if and only if each $\pi_v$ is $\SL_n(F_v)$-distinguished. 
\end{theorem}
\begin{proof}
One direction is obvious, so we suppose that $\pi$ is locally distinguished. We can always suppose that $\widetilde{\pi}$ 
conjugate self-dual by Proposition \ref{proposition characterization of global distinguished L packets}.

The group $G_\pi$ is Galois stable by Proposition \ref{proposition elementary}. As in \cite[Theorem 7.1]{ap13}, we need to prove that the group \[(H_1 \cap H_2)/H_3\] is trivial. In order to show that $H_1 \cap H_2 \subseteq H_3$, we claim that $H_2 \cap \A_F^\times \subseteq H_3$. 

So let $x \in E^\times G_\pi \cap \A_F^\times$. Note that $x^2=x x^\theta$, as $x \in \A_F^\times$. Since $G_\pi$ is Galois stable, we see that $x^2 \in F^\times G_\pi = H_3$. Indeed, writing $x=hk$, $h \in E^\times, k \in G_\pi$, we get
\[x^2=xx^\theta= hk h^\theta k^\theta = hh^\theta kk^\theta \in F^\times G_\pi.\]
Also $x^r \in G_\pi$ by Lemma \ref{lemma x to r fixes pi}. We have thus shown that both $x^2$ and $x^r$ are in $H_3$. It follows that $x \in H_3$, as $r$ is odd. 
\end{proof}

\begin{remark}
The simplifying role played by the fact that $r$ is odd in the proof of Theorem \ref{theorem localglobal} is quite analogous to its role in the proof of local multiplicity one, when $n$ is odd, for the pair $(\SL_n(E),\SL_n(F))$ (see \cite[p. 183]{ana05} or \cite[p. 1703]{ap18}).
\end{remark}

\section{Higher multiplicity for $\SL_n$}\label{section-higher}

We now suppose $n\geq 3$ and recall consequences of the works of Blasius, Lapid and Hiraga-Saito \cite{bla94,lap98,lap99,hs05,hs12}. This section contains no original result. 

\subsection{Different notions of multiplicity}\label{section various multiplicities}
Let $\pi$ be a cuspidal automorphic representation $\pi$ of $\SL_n(\A_E)$. We set 
\[m(\pi)=\dim \Hom_{\SL_n(\A_E)}(\pi,\mathcal{A}^{\infty}_0(\SL_n(E)\backslash \SL_n(\A_E))\] and call it the multiplicity of $\pi$ in the cuspidal spectrum of $\SL_n(\A_E)$.
There are several other notions of multiplicity for $\pi$, both on the automorphic side and on the Galois parameter side of the putative global Langlands correspondence. We shall need to pass from one to another and we explain the process in this paragraph. We follow \cite[p. 293]{lap98} and \cite[p. 162]{lap99}. First we consider the automorphic side. Thus, let $\widetilde{\pi}$ and $\widetilde{\pi}'$ be two cuspidal representations of $\GL_n(\A_E)$. We write:
\begin{enumerate}
\item[(i)] $\widetilde{\pi} \sim_s \widetilde{\pi}'$ if $\widetilde{\pi} \simeq \widetilde{\pi}' \otimes \eta$ for a Hecke character $\eta$ of $\A_E^\times$, 
\item[(ii)] $\widetilde{\pi} \sim_{ew} \widetilde{\pi}'$ if $\widetilde{\pi}_v \simeq \widetilde{\pi}_v' \otimes \eta_v$ for a character $\eta_v$ of $E_v^\times$ at each place $v$ of $E$
\item[(iii)] $\widetilde{\pi} \sim_w \widetilde{\pi}'$ if $\widetilde{\pi}_v \simeq \widetilde{\pi}_v' \otimes \eta_v$ for a character $\eta_v$ of $E_v^\times$ for almost places $v$ of $E$.
\end{enumerate}
One denotes by $M(\La(\widetilde{\pi}))$ the number of $\sim_s$ equivalence classes in the $\sim_{ew}$ equivalence class of $\widetilde{\pi}$, and by $\mathcal M(\La(\widetilde{\pi}))$ the number of $\sim_s$ equivalence classes in the $\sim_{w}$ equivalence class of $\widetilde{\pi}$. It was expected by Labesse and Langlands (\cite{ll79}) that if $\pi$ is a cuspidal automorphic representation of $\SL_n(\A_E)$ contained in $\La(\widetilde{\pi})$, 
then its multiplicity $m(\pi)$ inside the cuspidal automorphic spectrum is equal to $M(\La(\widetilde{\pi}))$ so that in particular $M(\La(\widetilde{\pi}))$ is finite. This was proved for $\SL_2(\A_E)$ in \cite{ll79} and in general for $\SL_n(\A_E)$ by Hiraga and Saito \cite[Theorem 1.6]{hs12}.

On the other hand the multiplicity $\mathcal M(\La(\widetilde{\pi}))$, which is conjectured to be finite and bounded by a function of $n$ in \cite[Conjecture 1]{lap99}, is certainly at least equal to $M(\La(\widetilde{\pi}))$ by definition, and related to a similar multiplicity on the ``Galois parameter side". To this end we introduce equivalence  relations $\sim_s$ and $\sim_w$ on the set of representations of a group $G$. Let $\phi$ and $\phi'$ be two 
morphisms from $G$ to $\mathrm{GL}_n(\C)$, we write:
\begin{enumerate}
\item[(i)] $\phi \sim_s \phi'$ if there is $x\in \mathrm{PGL}_n(\C)$ such that $\overline{\phi'(g)}=x^{-1}\overline{\phi(g)} x\in \mathrm{PGL}_n(\C)$ for all $g\in G$, in which case we say that $\phi$ and $\phi'$ are strongly equivalent.
\item[(ii)] $\phi \sim_w \phi'$ if for all $g\in G$, there is $x_g\in \mathrm{PGL}_n(\C)$ such that $\overline{\phi'(g)}=x_g^{-1}\overline{\phi(g)} x_g\in \mathrm{PGL}_n(\C)$, in which case we say that $\phi$ and $\phi'$ are weakly equivalent.
\end{enumerate}
We denote by $\mathcal M(\phi)$ the number of $\sim_s$ equivalence classes in the $\sim_{w}$ equivalence class of $\phi$. One of the main achievements of \cite{lap98,lap99} is the following result (cf. \cite[Theorem 6]{lap98} and \cite[Theorem 2]{lap99}).

\begin{theorem}\label{theorem Lapid multiplicity}
Let $L$ be a Galois extension of $E$ with respective Weil groups $W_L$ and $W_E$ such that $\Gal(L/E)$ is nilpotent, and let $\chi$ be a Hecke character of $\A_E^\times$ such that $\phi=\Ind_{W_L}^{W_E}(\chi)$ is irreducible. Denote by 
$\widetilde{\pi}=\widetilde{\pi}(\phi)$ the cuspidal automorphic representation of $\GL_n(\A_E)$ associated to $\Ind_{W_E}^{W_F}(\chi)$ by \cite{ac89}. Then $\mathcal M(\phi)=\mathcal M(\La(\widetilde{\pi}))$.
\end{theorem}

\begin{remark} \label{remark Lapid Galois type} In the proof of this result Lapid invokes the Chebotarev density theorem to argue that for such representations, the relations $\sim_s$ and $\sim_w$ are compatible on the Galois parameter side and the automorphic side, and shows that if 
$\widetilde{\pi'}\sim_w \widetilde{\pi}$ (i.e., almost everywhere a twist of $\widetilde{\pi}$) for $\widetilde{\pi}$ as in 
the statement of Theorem \ref{theorem Lapid multiplicity}, then $\widetilde{\pi}'$ is of Galois type, i.e., there exists a Galois representation $\phi'$, necessarily unique, of $W_E$ with Satake parameters equal to those of $\widetilde{\pi}'$ at almost every place of $E$. We shall use these facts as well in what follows.
\end{remark}

\begin{remark} \label{remark all equal multiplicities} In particular suppose that $\widetilde{\pi}$ and $\mathcal M(\phi)$ are as in 
the statement of Theorem \ref{theorem Lapid multiplicity}, and suppose moreover that the weak equivalence class of $\widetilde{\pi}$ (its $\sim_w$ class) is the same as its $\sim_{ew}$ class, then for any $\pi \in \La(\widetilde{\pi})$, we have:
\[m(\pi)= M(\La(\widetilde{\pi}))= \mathcal M(\La(\widetilde{\pi}))=\mathcal M(\phi).\]
Note that the middle equality can in general be a strict inequality, see for example \cite[Proposition 2.5]{bla94}.
\end{remark}

\subsection{Examples of higher cuspidal multiplicity due to Blasius}\label{blasius}

In this section we recall the first fundamental construction, due to D. Blasius (\cite{bla94}), of representations appearing with a multiplicity greater than one in the cuspidal spectrum of $\SL_n(\A_E)$. In view of the more recent results of Lapid and Hiraga-Saito recalled 
in Section \ref{section various multiplicities}, we give a slightly more modern treatment of the construction of Blasius, however following its exact same lines. For $p$ a fixed prime number, we denote by $H_p$ the Heisenberg subgroup of $\GL_3(\F_p)$ of upper triangular unipotent matrices with order $p^3$. Blasius considers finite products of Heisenberg groups \[H_{p_i}=\left\lbrace \left( \begin{array}{ccc} 1 & a & c \\ 0 & 1 & b \\ 0 & 0 & 1 \end{array}\right) \mid \ a,\ b, \ c \in \Z/p_i \right\rbrace,\] where for our purpose we restrict a finite number of odd primes $p_i$ possibly equal for $i\neq j$. For each index $i$, we denote by $Z_i$ the center of $H_{p_i}$, and by $\mathcal{L}_i$ the Lagrangian subgroup of $H_{p_i}$ given by $a=0$. We then set $H=\prod_i H_{p_i}$, $\mathcal{L}=\prod_i \mathcal{L}_i$ and $Z=\prod_i Z_i$.

Now let $E$ be our number field. Since $H$ is a product of $p$-groups it is solvable, and therefore by the well-known result of Shafarevich in inverse Galois theory, there is a Galois extension 
$L/E$ such that $\Gal(L/E)=H$. Now take for each $i$ a non-trivial character $\chi_i$ of $Z_i$ and extend $\chi_i$ to a character $\widetilde{\chi_i}$ of $\mathcal{L}_i$ by \[\widetilde{\chi_i}\left( \begin{array}{ccc} 1 & 0 & c \\ 0 & 1 & b \\ 0 & 0 & 1 \end{array}\right)=\chi_i \left( \begin{array}{ccc} 1 & 0 & c \\ 0 & 1 & 0 \\ 0 & 0 & 1 \end{array}\right).\] Now set $\chi=\otimes_i \chi_i$ the corresponding character of $Z$, and call it a \emph{regular} character of $Z$ (meaning all the $\chi_i$ are non-trivial), and  
$\widetilde{\chi}=\otimes_i \widetilde{\chi_i}$ to be the corresponding character of  
$\mathcal{L}=\Gal(L/L_{\mathcal{L}})$ (for $L_{\mathcal{L}}$ an extension of $E$). This character can be seen as a Hecke character of the Weil group $W_{L_{\mathcal{L}}}$ (which is trivial on $W_L$). The induced representation 
$I_{\chi}=\Ind_{W_{L_{\mathcal{L}}}}^{W_E}(\widetilde{\chi})$ is an irreducible representation of $H$ of dimension $n=\prod_i p_i$ and 
when $\chi$ varies, the representations $I_\chi$ are non-isomorphic and describe all the irreducible representations of $H$, their number being equal to 
\[m(n)=\prod_i (p_i-1).\] We then set $\widetilde{\pi}_{\chi}$ to be the cuspidal automorphic representation of $\GL_n(\A_E)$ attached to $I_{\chi}$ in \cite{ac89}. By Theorem 
\ref{theorem Lapid multiplicity} we obtain the following result from Section 1.1 of \cite{bla94}.

\begin{proposition}\label{proposition Lapid Theorem 2}
In the situation above, let $\pi\in \mathcal{A}_0^\infty(\SL_n(E)\backslash \SL_n(\A_E))$ be an irreducible summand of $\widetilde{\pi}_{\chi}$. Then $\mathcal M(\La(\widetilde{\pi}_{\chi}))=m(n)$. 
\end{proposition}
\begin{proof}
According to Theorem \ref{theorem Lapid multiplicity}, it is sufficient to check that the conjugacy class of $I_\chi(w)$ in ${\rm PGL}_n(\C)$ is independent of $\chi$ for any $w\in W_E$ but that the $I_{\chi}$'s are inequivalent projective representations. This is done in 
\cite[Section 1.1]{bla94}. 
\end{proof}

We are however looking for information on $m(\pi)$ rather than $\mathcal M(\La(\widetilde{\pi}_{\chi}))$. Therefore we follow Blasius again to put us in a situation where $\mathcal M(\La(\widetilde{\pi}_{\chi}))= M(\La(\widetilde{\pi}_{\chi}))$ in order to apply Remark \ref{remark all equal multiplicities}. To this end we select $L$ as in the proof of \cite[Proposition 2.1]{bla94}, such that at all the places in $L$ lying above $p$ for each $p$ dividing $|H|$, $L$ is unramified. 

Then in such a situation, by  \cite[Proposition 2.1, (2)]{bla94}, we deduce that two representations $\widetilde{\pi}_{\chi}$ and $\widetilde{\pi}_{\chi'}$, for regular characters $\chi$ and $\chi'$ of $Z$, are not only weakly equivalent (which we already know from \cite[Section 1.1]{bla94} and Section \ref{section various multiplicities}), but they are in fact in the same $\sim_{ew}$-class, i.e., they are twists of each other at every place of $E$. Finally, by Remark \ref{remark Lapid Galois type}, if 
$\widetilde{\pi}$ is a cuspidal automorphic representation of $\GL_n(\A_E)$ weakly equivalent to $\widetilde{\pi}_{\chi}$, it is of Galois type with Galois parameter say $\phi$. Because for every $w\in W_E$, the conjugacy class of $I_{\chi}(w)$ in $\GL_n(\C)$ is equal to that of $\phi(w)$, we deduce that $I_\chi$ and $\phi$ have the same kernel, and are thus in fact both irreducible representations of $H$. This implies that $\phi$ is itself of the form $I_{\chi'}$ for a regular character $\chi'$ of $Z$, in particular the $\sim_w$ class of $\pi$ is equal to its $\sim_{ew}$ class. In view of Remark \ref{remark all equal multiplicities}, the outcome of this discussion is the following result, which also follows from the proof of \cite[Proposition 3.3]{bla94}.

\begin{proposition}\label{proposition Lapid Theorem 2}
Let $E$ be a number field and let $L$ be an extension of $E$ such that $\Gal(L/E)\simeq H$ and such that $L$ is unramified at every place of $L$ lying over a prime divisor of the cardinality $n=|H|$. Let $\chi$ be a regular character of $Z$ and let $\pi\in \mathcal{A}_0^\infty(\SL_n(E)\backslash \SL_n(\A_E))$ be an irreducible summand of $\widetilde{\pi}_{\chi}$. Then $m(\pi)=m(n)$ and the $\La$-packets containing a copy of $\pi$ are those of the form $\La(\widetilde{\pi}_{\chi'})$ for a regular character $\chi'$ of $Z$ and they are all different.
\end{proposition}

\begin{remark}
Such extensions $L$ of $E$ exist in abundance by Shafarevich's theorem in inverse Galois theory.
\end{remark}

\begin{remark}\label{hs05}
In \cite{bla94}, Blasius had conjectured that two $\La$-packets, say $\La(\widetilde{\pi})$ and $\La(\widetilde{\pi}')$, would be isomorphic if $\widetilde{\pi}$ and $\widetilde{\pi}'$ are locally isomorphic at every place up to a character twist \cite[Conjecture on p. 239]{bla94}. This conjecture was later proved by Hiraga-Saito in 2005 \cite{hs05}. Lacking the truth of the conjecture at that point in time, \cite{bla94} resorted to a trick using complex conjugation. Note that reading out the precise multiplicity $m(\pi)$ is an immediate consequence of this result.
\end{remark}

\section{Two questions}\label{questions}

In this section we attempt to two natural and important questions. We thank Rapha\"el Beuzart-Plessis and Dipendra Prasad for posing the first of these questions to us in the context of this paper. We then consider one more question which in the case of $\SL(2)$ was answered by an explicit construction in \cite[Theorem 8.2]{ap06}. The key ingredient in all our constructions is the explicit nature of the examples of cuspidal representations of high multiplicity in \cite{bla94,lap98,lap99}. In these examples, we also need to make a crucial use of the main result of this paper (cf. Theorem \ref{theorem main}). 

\subsection{Questions}\label{qns}

We formulate two natural questions for each of which we provide answers in the later subsections.

\begin{question}\label{qn2}
Consider the natural decomposition of $\mathcal{A}^{\infty}_0(\SL_n(E)\backslash \SL_n(\A_E))$ into $\La$-packets. Let $\pi_1$ and $\pi_2$ be two irreducible submodules of $\mathcal{A}^{\infty}_0(\SL_n(E)\backslash \SL_n(\A_E))$ such that $\pi_1\simeq \pi_2$ but which belong to two different $\La$-packets $\La(\widetilde{\pi}_1)\neq \La(\widetilde{\pi}_2)$. If $p_n$ does not vanish on $\pi_1$, then is it true that it does not vanish on $\pi_2$?
\end{question}

\begin{remark}\label{remark 2} We shall see in Section \ref{examples} that the answer is No in general. Thus, for $n\geq 3$, there are 
cuspidal automorphic representations of $\SL_n(\A_E)$ which are locally distinguished, but with at least one  canonical realization in the space of smooth cusp forms on which $p_n$ vanishes.
\end{remark}

The following question arises immediately after the above remark.

\begin{question}\label{qn3}
For $n\geq 3$, are there cuspidal automorphic representations of $\SL_n(\A_E)$ which are locally distinguished at every place of $F$, but not globally? In fact is it even possible to construct such a representation which belongs to no distinguished $\La$-packet? \end{question}

We shall see in Section \ref{subsection examples 3} that such representations do exist. Note that though Question \ref{qn2} is not meaningful for $\SL_2(\A_E)$ according to Ramakrishnan's multiplicity $1$ result \cite{ram00}, the issues addressed by Remark \ref{remark 2}, as well as Question \ref{qn3} make sense for $n=2$. In this case they are all answered in \cite{ap06}. In fact it is sufficient to answer Question \ref{qn3} for $n=2$, and this is done by \cite[Theorem 8.2]{ap06}, the proof of which is quite involved: there are indeed cuspidal automorphic representations of $\SL_2(\A_E)$ which are locally distinguished at every place of $F$ but not globally. We shall provide easier examples of this type in Section \ref{examples} for $n\geq 3$. 

\subsection{Distinguished cuspidal representations of higher multiplicity}\label{bp}

Now we need to construct cuspidal representations $\pi$ of $\SL_n(\A_E)$ which are $\SL_n(\A_F)$-distinguished with $m(\pi) \geq 2$ for odd $n$. 

Let us explain our general recipe for this, using the examples of Blasius in \S \ref{blasius}. We take $n\geq 3$ odd and write it as $n = \prod_i p_i$. We set 
$H = \prod_i H_{p_i}$ as before and take an involution $\sigma$ of the group $H$. Associated to this involution is the semi-direct product 
\[G = H \rtimes \Z/2\]
where $\Z/2$ acts on $H$ via $\sigma$. Now let $F$ be any number field and let $L$ be an extension of $F$ such that Gal$(L/F) \simeq G$. In fact we choose $L$ in such a way that $L/F$ is unramified at each place of $F$ lying above any $p$ dividing $n$. Note that all these can be done by Shafarevich's theorem since $G$ is solvable.  Let $E$ be the fixed field of $H$ so that 
\[{\rm Gal}(L/E) \simeq H ~~ \& ~~ {\rm Gal}(E/F) = \langle \sigma \rangle.\]
Take an irreducible representation $\rho$ of $H$. It identifies with $I_{\chi_{\rho}}$ for $\chi_{\rho}$ a regular character of $Z$ and we set $\widetilde{\pi}(\rho)=\widetilde{\pi}_{\chi_{\rho}}$ (cf. \S \ref{blasius}). In particular, because $L/E$ is unramified at places of $E$ lying above the prime divisors of $n$, if $\pi$ belongs to $\La(\widetilde{\pi}(\rho))$, we obtain $m(\pi)=m(n)$ thanks to Proposition \ref{proposition Lapid Theorem 2}. In this situation, we have the following very useful result due to the rigidity of the representation theory of Heisenberg groups, which we will apply in order to produce examples answering Question \ref{qn2}:

\begin{proposition}\label{proposition recipe for general examples} 
In the situation described above, take an irreducible representation $\rho$ of $H$ and denote by $c_\rho$ its central character. The $\La$-packet $\La(\widetilde{\pi}(\rho))$ is distinguished if and only if 
$c_\rho(z^\sigma) =c_\rho(z^{-1})$ for all $z\in Z$.
\end{proposition} 
\begin{proof}
By Proposition \ref{proposition characterization of global distinguished L packets}, $\La(\widetilde{\pi}(\rho))$ is distinguished if and only if 
$(\widetilde{\pi}(\rho)^\vee)^\sigma \simeq \mu \otimes \widetilde{\pi}(\rho)$ for a Hecke character $\mu$ factoring through 
$N_{E/F}$. This is equivalent to $\widetilde{\pi}((\rho^\vee)^\sigma) \simeq \mu \otimes \widetilde{\pi}(\rho)$. However as 
the $\La$-packets determined by different irreducible representations are different thanks to Proposition \ref{proposition Lapid Theorem 2}, we easily deduce that $\La(\widetilde{\pi}(\rho))$ is distinguished if and only if $\rho$ is conjugate self-dual, i.e., $\rho^\vee \simeq \rho^\sigma$. The result now follows from the fact that $\rho$ is determined by its central character.
\end{proof}

In view of Corollary \ref{corollary crux}, a consequence of Proposition \ref{proposition recipe for general examples} is 
the following.

\begin{corollary}\label{corollary recipe for general examples} 
In the situation of Proposition \ref{proposition recipe for general examples}, let $\rho$ be an irreducible representation of $H$ such that $c_\rho^\sigma=c_\rho^{-1}$, and $\pi\in \La(\widetilde{\pi}(\rho))$ such that $\mathcal{P}_{\SL_n(\A_F)}$ does not vanish on $\pi$. Then the canonical copies of $\pi$ on which $\mathcal{P}_{\SL_n(\A_F)}$ does not vanish are those contained in the 
$\La$-packets of the form $\La(\widetilde{\pi}(\rho'))$ with $\rho'$ an irreducible representation of $H$ such that 
$c_{\rho'}^\sigma=c_{\rho'}^{-1}$.
\end{corollary}

\subsection{Examples for Question \ref{qn2}}\label{examples}

We first give two examples for which we answer Question \ref{qn2}. In the first one, 
all the canonical copies of the considered distinguished representation have a non-vanishing period, whereas in the second example only some of the canonical copies of the considered distinguished representation have a non-vanishing period and some others do not have a non-vanishing period.

For the first set of examples, the group $H$ is as in Section \ref{bp} and the involution that we consider on it, for $a, \ b$ and $c$ in $\prod_i \Z/p_i$, is given by 
\[\sigma: \left( \begin{array}{crr} 1 & a & c \\ 0 & 1 & b \\ 0 & 0 & 1 \end{array} \right) \mapsto \left( \begin{array}{crr} 1 & a & -c \\ 0 & 1 & -b \\ 0 & 0 & 1 \end{array} \right).\]
In this case because the associated involution acts as the inversion on $Z$, Proposition \ref{proposition recipe for general examples} tells us that all $\La$-packets 
$\La(\widetilde{\pi}(\rho))$ are distinguished when $\rho$ varies in the set of irreducible classes of representations of $H$, and that if one fixes a representation $\pi$ in one $\La$-packet on which $\mathcal{P}_{\SL_n(\A_F)}$ does not vanish, then it does not vanish on any of the $m(n)$ canonical copies of $\pi$.

For the second set of examples, we consider $H$ as above (of odd cardinality $n$) and $H'=H\times H$ (which is in fact a special type of $H$) endowed with the switching involution 
\[\sigma:(x,y) \mapsto (y,x).\]
In this case Proposition \ref{proposition recipe for general examples} tells us that the distinguished $\La$-packets of 
$\SL_{n^2}(\A_E)$ of the form $\La(\widetilde{\pi}(\rho'))$ are the $m(n)$ ones such that $\chi_{\rho'}$ is of the form $\chi\otimes \chi^{-1}$ with $\chi$ regular, whereas the others are not. Then again by Corollary \ref{corollary recipe for general examples} we conclude that if 
$\pi$ is a fixed distinguished representation of $\SL_{n^2}(\A_E)$ appearing in one of the $m(n)^2$ many $\La$-packets above, then the period $\mathcal{P}_{\SL_{n^2}(\A_F)}$  does not vanish on the $m(n)$ canonical copies inside the distinguished $m(n)$ many distinguished $\La$-packets, and does vanish on the $m(n)^2-m(n)$ remaining ones.

\subsection{Examples for Question \ref{qn3}}\label{subsection examples 3} 

Now we give a set of examples answering Question $3$, using again Proposition \ref{proposition recipe for general examples}.  

For simplicity we take $H=H_p$ for 
$p$ an odd prime (i.e., $n=p$),  and we also take $L/F$, hence in particular $E/F$, split at Archimedean places (however we explain in Remark \ref{remark distinguished ps} how to get rid of this assumption). Let $\sigma$ be an involution of $H$ such that $z^\sigma = z$ for all $z \in Z$. Thus, we may take the trivial involution or the involution of $H$ given by 
\[\sigma: \left( \begin{array}{crr} 1 & a & c \\ 0 & 1 & b \\ 0 & 0 & 1 \end{array} \right) \mapsto \left( \begin{array}{crr} 1 & -a & c \\ 0 & 1 & -b \\ 0 & 0 & 1 \end{array} \right).\]
Since $z^\sigma = z$ for all $z\in Z$, 
Proposition \ref{proposition recipe for general examples} implies that no $\La$-packet of the form 
$\widetilde{\pi}(\rho)$ for $\rho$ an irreducible representation of $H$ is distinguished because, as $|Z|$ is odd, the only character of $Z$ of order $\leq 2$ is trivial.

It remains to prove that if we fix $\rho$ as above, and set $\widetilde{\pi}=\widetilde{\pi}(\rho)$, then $\La(\widetilde{\pi})$ contains an automorphic
representation $\pi$ such that $\pi_v$ is $\SL_p(F_v)$-distinguished for every place $v$ of $F$. This is equivalent to showing that $\widetilde{\pi}_v$ is $(\GL_p(F_v),\gamma_v)$-distinguished for some character $\gamma_v$ of $F_v^\times$, which is what we do. Recall that by \cite[Proposition 2.1]{bla94},
\[\widetilde{\pi}_v^\sigma \simeq \widetilde{\pi}_v^\vee \otimes \eta_v\]
at each place $v$ for a character $\eta_v$ of $E_v^\times$. 

If a place $v$ of $F$ splits in $E$ as $(v_1,v_2)$ then the above condition implies $\widetilde{\pi}_v$ is of the form $(\tau,\tau^\vee \otimes \nu)$ which is distinguished for the character $\nu$ of $\F_v^\times$. 

Now let $v$ be such that it does not split in $E$, in particular $v$ is finite. We set $B_p(E_v)$ the upper triangular Borel subgroup of $\GL_p(E_v)$.

We write as before 
$\widetilde{\pi}=\widetilde{\pi}(\rho)$ for $\rho$ an irreducible representation of $H$. We denote by $\mathcal L$ and $\mathcal{L}'$ the first and the second  Lagrangian subgroups of $H$ given by $a=0$ and $b=0$ respectively (cf. \S \ref{blasius}). 
By the proof of \cite[Proposition 2.1]{bla94} the local Galois group of $H_v$ is an abelian subgroup of $H$, hence either trivial, equal to $Z$, $\mathcal{L}$ or $\mathcal{L}'$. We recall that $\rho=\Ind_{\mathcal{L}}^H(\widetilde{\chi})$ where 
\[\widetilde{\chi} \left( \begin{array}{crr} 1 & 0 & c \\ 0 & 1 & b \\ 0 & 0 & 1 \end{array} \right) = \chi(c)\]
for $\chi$ a non-trivial character of $\Z/p$. We fix $\mu$ a non-trivial character $\Z/p$ and set 
\[\widetilde{\mu} \left( \begin{array}{crr} 1 & 0 & c \\ 0 & 1 & b \\ 0 & 0 & 1 \end{array} \right) = \mu(b).\]
Similarly we set
\[\widetilde{\chi}' \left( \begin{array}{crr} 1 & a & c \\ 0 & 1 & 0 \\ 0 & 0 & 1 \end{array} \right) = \chi(c)\]
and 
\[\widetilde{\mu}' \left( \begin{array}{crr} 1 & a & c \\ 0 & 1 & 0 \\ 0 & 0 & 1 \end{array} \right) = \mu(a).\]

Clearly if $H_v$ is trivial or equal to $Z$, then $\rho_{|H_v}$ is a sum of copies of the same character, hence $\widetilde{\pi}_v$ is of the form 
\[\mathrm{Ps}(\alpha,\dots,\alpha)=\Ind_{B_p(E_v)}^{\GL_p(E_v)}(\alpha\otimes \dots \otimes \alpha),\] where induction is normalized, hence 
$\alpha_{|F_v^\times}$-distinguished by \cite[Theorem 5.2]{mat11}. Now we consider the case 
$H_v=\mathcal{L}$. Then by Mackey theory,
\[\rho_{|\mathcal{L}}=\widetilde{\chi}.(\oplus_{k=0}^{p-1}\widetilde{\mu}^k).\] 
Thus the corresponding principal series is of the form
\begin{equation}\label{1}
\widetilde{\pi}_v = \mathrm{Ps}(\alpha,\alpha\beta,\alpha\beta^{-1},\dots,\alpha\beta^{(p-1)/2},\alpha\beta^{-(p-1)/2}).
\end{equation}

If $\sigma$ is the trivial involution we trivially have $\beta = \beta^\sigma$ so (\ref{1}) takes the form
\[\widetilde{\pi}_v = \alpha \otimes \mathrm{Ps}(1,\beta,\beta^{-\sigma},\dots,\beta^{(p-1)/2},(\beta^{(p-1)/2})^{-\sigma}),\]
which is distinguished by \cite[Theorem 5.2]{mat11}.

If $\sigma$ is the non-trivial involution such that $z^\sigma = z$ for $z \in Z$ then note that $\sigma$ fixes $\widetilde{\chi}$ whereas it sends $\widetilde{\mu}$ to its inverse. We set 
$\mu_k=\alpha\beta^{k}$ for $k=1,\dots,(p-1)/2$, so that (\ref{1}) takes the form
\[\widetilde{\pi}_v= \mathrm{Ps}(\alpha,\mu_1,\mu_1^\sigma,\dots,\mu_{(p-1)/2},\mu_{(p-1)/2}^\sigma).\] 
Now because for $k=1,\dots,(p-1)/2$, one has $\alpha^2=\mu_k\mu_k^\sigma$ and hence $\alpha_{|F_v^\times}^2={\mu_k}_{|F_v^\times}^2$, but as both characters in this equality have odd order $p$ we deduce that $\alpha_{|F_v^\times}={\mu_k}_{|F_v^\times}$. So 
\[\widetilde{\pi}_v= \alpha\otimes \mathrm{Ps}(1,\alpha^{-1}\mu_1,\alpha^{-1} \mu_1^\sigma,\dots,\alpha^{-1} \mu_{(p-1)/2},\alpha^{-1} \mu_{(p-1)/2}^\sigma)\] and all the characters appearing in the principal series have trivial restriction to $F_v^\times$, we deduce again from \cite[Theorem 5.2]{mat11} that $\widetilde{\pi}_v$ is $\alpha_{|F_v^\times}$-distinguished. 

Finally when $H_v=\mathcal{L}'$ then 
\[\rho_{|\mathcal{L}'}=\widetilde{\chi}'.(\oplus_{k=0}^{p-1}\widetilde{\mu}'^k)\] and a completely similar argument proves that 
$\widetilde{\pi}_v$ is distinguished by a character.

Hence $\La(\widetilde{\pi})$ does not contain any distinguished representation but it contains cuspidal representations which are everywhere locally distinguished.

\begin{remark}\label{remark distinguished ps}
In constructing examples in this section, we chose $L/F$ such that the Archimedean places split in order to have $E/F$ split at the Archimedean places. This assumption can be removed as the characterisation of a generic distinguished principal series as in \cite[Theorem 5.2]{mat11} is true also for $(\GL_n(\mathbb C),\GL_n(\mathbb R))$. Namely a generic principal series 
$\mathrm{Ps}(\chi_1,\dots,\chi_n)$ of $\GL_n(\C)$ is $\GL_n(\R)$-distinguished if and only if there is an involution $\epsilon$ of in the symmetric group $S_n$ such that $\chi_{\epsilon(i)}=\chi_i^{-\sigma}$ for any $i=1,\dots,n$, and moreover 
$(\chi_i)_{|\R^\times}=1$ if $\epsilon(i)=i$. The direct implication is a special case of \cite[Theorem 1.2]{kem15}, whereas the other implication can be obtained as follows. 
First up to re-ordering (which is possible as the principal series is generic by assumption) we can suppose that there is $1\leq s\leq \lfloor n/2 \rfloor$ such that $\chi_{2i}=\chi_{2i-1}^{-\sigma}$ for $i=1,\dots,s$, and that $(\chi_i)_{|\R^\times}=1$ for $i=2s+1,\dots,n$. Now a principal series $\mathrm{Ps}(\chi,\chi^{-\sigma})$ of $\GL_2(\C)$ is $\GL_2(\R)$-distinguished. Indeed by \cite[Théorème 3]{cd94}, for $s\in \C$ with $\mathrm{Re}(s)$ large enough, there is a $\GL_2(\R)$-invariant continuous linear form $L_s$ on $\mathrm{Ps}(\chi| \ . \ |_{\R}^s,\chi^{-\sigma}| \ . \ |_{\R}^{-s})$, and a holomorphic function $h$ on $\C$ such that $h(s)L_s(f_s)$ extends to a holomorphic function on $\C$ for any flat section $f_s$ of $\mathrm{Ps}(\chi| \ . \ |_{\R}^s,\chi^{-\sigma}| \ . \ |_{\R}^{-s})$. Moreover by [ibid.] the function  $h(s)L_s(f_s)$ is non-zero for some choice of $f_s$, which we can suppose to be $\mathrm{U}(2,\C/\R)$-finite because $L_s$ is continuous for $\mathrm{Re}(s)$ large enough. A standard leading term argument then allows to regularize $L_s$ at $s=0$ to define a non-zero $\GL_2(\R)$-invariant linear form $L$ on the space of $\mathrm{U}(2,\C/\R)$-finite vectors in $\mathrm{Ps}(\chi,\chi^{-\sigma})$. Finally one extends $L$ to a necessarily non-zero element of $\Hom_{\GL_2(\R)}(\mathrm{Ps}(\chi,\chi^{-\sigma}),\C)$ by \cite[Théorème 1]{bd92}. Once we have this result, the transitivity of parabolic induction together with a closed orbit contribution argument allows to define a non-zero $\GL_n(\R)$-invariant linear form on $\mathrm{Ps}(\chi_1,\dots,\chi_n)$. 
\end{remark}

\begin{remark}
It is not hard to extend the examples obtained in this section in the cuspidal case, to the square-integrable case, using the results of this paper. 
\end{remark}

\section*{Acknowledgements}

The authors thank Rapha\"el Beuzart-Plessis and Yiannis Sakellaridis for useful comments and explanations. The content of \S \ref{bp} \& \ref{examples} grew out of a discussion with Beuzart-Plessis. The authors would like to especially thank Dipendra Prasad for his questions and comments over several e-mail conversations; his guidance in general has played a significant role in the writing of some parts of this paper.


\end{document}